\documentclass[11pt,a4paper]{article}
\usepackage{amsthm}	
\usepackage{amsfonts, amssymb, amsmath, amscd, latexsym}
\usepackage{eucal, enumerate, latexsym}
\usepackage{graphics}
\usepackage{graphicx}
\usepackage{mathscinet}
\usepackage{psfrag}
\usepackage{tikz}
\usetikzlibrary{decorations.markings}

\usepackage{hyperref}

\newcommand{\R}{\mathbb R}
\newcommand{\N}{\mathbb N}
\newtheorem{theorem}{Theorem}[section]
\newtheorem{corollary}{Corollary}[section]
\newtheorem{proposition}{Proposition}[section]
\newtheorem{definition}{Definition}[section]
\newtheorem{remark}{Remark}[section]
\newtheorem{lemma}{Lemma}[section]
\newtheorem{example}{Example}[section]

\newcommand{\ds}{\displaystyle}

\renewcommand{\epsilon}{\varepsilon}

\renewcommand{\phi}{\varphi}

\newcommand{\fs}{\footnotesize}

\makeatletter

\newlength{\captionwidth}
\long\def\@makecaption#1#2{%
   \vskip 10\p@
   \setbox\@tempboxa\hbox{#1: #2}%
   \ifdim \wd\@tempboxa > \captionwidth 
       \hbox to\hsize{\hfil
       \parbox[t]{\captionwidth}{
       \small#1: \small#2\par}
       \hfil}
     \else
       \hbox to\hsize{\hfil\box\@tempboxa\hfil}%
   \fi}

\makeatother

\usetikzlibrary{decorations.pathreplacing,decorations.markings}
\tikzset{
  on each segment/.style={
    decorate,
    decoration={
      show path construction,
      moveto code={},
      lineto code={
        \path [#1]
        (\tikzinputsegmentfirst) -- (\tikzinputsegmentlast);
      },
      curveto code={
        \path [#1] (\tikzinputsegmentfirst)
        .. controls
        (\tikzinputsegmentsupporta) and (\tikzinputsegmentsupportb)
        ..
        (\tikzinputsegmentlast);
      },
      closepath code={
        \path [#1]
        (\tikzinputsegmentfirst) -- (\tikzinputsegmentlast);
      },
    },
  },
  mid arrow/.style={postaction={decorate,decoration={
        markings,
        mark=at position .6 with {\arrow[#1]{stealth}}
      }}},
}

\begin{document}

\title{Diffusion-convection reaction equations
with sign-changing diffusivity and bistable reaction term}

\author{
Diego Berti\footnote{Department of Sciences and Methods for Engineering, University of Modena and Reggio Emilia, Italy}
\and
Andrea Corli\footnote{Department of Mathematics and Computer Science, University of Ferrara, Italy}
\and
Luisa Malaguti\footnotemark[1]
}



\maketitle
\begin{abstract}
We consider a reaction-diffusion equation with a convection term in one space variable, where the diffusion changes sign from the positive to the negative and the reaction term is bistable. We study the existence of wavefront solutions, their uniqueness and regularity. The presence of convection reveals several new features of wavefronts: according to the mutual positions of the diffusivity and reaction, profiles can occur either for a single value of the speed or for a bounded interval of such values; uniqueness (up to shifts) is lost; moreover, plateaus of arbitrary length can appear; profiles can be singular where the diffusion vanishes.

\vspace{1cm}

\end{abstract}

\section{Introduction}\label{s:I}

In this paper we deal with the parabolic equation
\begin{equation}\label{e:E}
\rho_t + f(\rho)_x=\left(D(\rho)\rho_x\right)_x + g(\rho), \qquad t\ge 0, \, x\in \R.
\end{equation}
Equation \eqref{e:E} arises as a simple model in several applications where a reaction (modeled by $g$) is propagated in the space both in a convective and in a diffusive way; the latter are modeled by $f$ and $D$, respectively. In these applications the function $\rho$ usually stands for a concentration or a density and takes value in a bounded interval; without loss of generality we assume that such an interval is $[0,1]$.
An interesting feature is that the diffusivity $D$ can be {\em negative}. This occurs, for instance, in the modeling of collective movements \cite{Nelson_2000, Whitham}, in oil recovery \cite{DJLW}, in thermodynamics \cite{Kerner-Osipov}, and in biology \cite{HPO, Turchin}.

We are interested in {\em traveling-wave} solutions (TWs for short) to \eqref{e:E}, i.e., in solutions to \eqref{e:E} of the form $\rho(t,x)= \phi(x-ct)$. The function $\phi=\phi(\xi)$ is the {\em profile} of the TW and the real number $c$ is its {\em speed}. The equation for the profile is then
\begin{equation}\label{e:ODE}
\left(D(\phi)\phi^{\prime}\right)^{\prime}+\left(c -h(\phi)\right)\phi' + g(\phi)=0,
\end{equation}
where ${'}$ stands for the derivative with respect to $\xi$; we denote the derivative of $f$ with respect to $\rho$ as $h(\rho)=\dot f(\rho)$. We refer to \cite{GK} for more information on traveling waves (see also \cite{Bonheure}).

The current paper carries on the analysis that we began in \cite{BCM1, BCM2}; the aim of these researches is to provide a general treatment of traveling-wave solutions to equations with sign-changing diffusivities, with source terms that possibly change sign, and in presence of convective terms. They were motivated by modeling of collective movements, see \cite{CM-ZAMP}, and in biology \cite{Maini-Malaguti-Marcelli-Matucci2006, Maini-Malaguti-Marcelli-Matucci2007}.

More precisely, in \cite{BCM1} we considered the case when both $D>0$ and $g>0$ in $(0,1)$, and $g$ vanishes either at $0$ or at both extrema; the latter is the so-called {\em monostable} case. The aim of \cite{BCM1} was to complete or improve previous results obtained in \cite{CdRM, CM-DPDE, MMconv} as well as to lay a background for the subsequent papers, in particular as far as non-uniqueness of solutions is concerned. The statement of a typical result in \cite{BCM1} is that profiles exist if and only if $c\ge c^*$, for a real threshold $c^*\in\R$.

The case when the diffusivity $D$ changes sign has been studied in \cite{BCM2} for a monostable reaction term $g$, see also \cite{Bao-Zhou2014, Bao-Zhou2017, Maini-Malaguti-Marcelli-Matucci2006}; equation \eqref{e:E} becomes then a forward-backward parabolic equation. There, we extended previous results, obtained only in the case $f=0$, and emphasized the much richer dynamics provided by the presence of the convective term $f$.

The aim of this paper is to allow also $g$ to change sign, and namely from the negative to the positive; this is the {\em bistable} (or Allen-Cahn) case.

Here follow our assumptions, see Figure \ref{f:fDpn-a>g}. On the convective term $f$ we assume
\begin{itemize}
\item[{(f)}]\, $f\in C^1[0,1]$, $f(0)=0$.
\end{itemize}
Clearly, the condition $f(0)=0$ just fixes a representative, because the convection is only defined up to an additive constant. The diffusivity $D$ satisfies, for some $\alpha\in (0,1)$:
\begin{itemize}
\item[(D)] \, $D\in C^1[0,1]$, $D>0$ in $(0,\alpha)$ and $D<0$ in $(\alpha,1)$.
\end{itemize}
Notice, condition (D) leaves open the possibility for $D$ to vanish at $0$ or $1$.
On the reaction term $g$ we assume, for some $\gamma \in(0,1)$,
\begin{itemize}
\item[{(g)}]\, $g\in C^0[0,1]$, $g<0$ in $(0,\gamma)$, $g>0$ in $(\gamma,1)$ and $g(0)=g(\gamma)=g(1)=0$.
\end{itemize}

\begin{figure}[htb]
\begin{center}

\begin{tikzpicture}[>=stealth, scale=0.41]
\draw[->] (0,0) --  (10,0) node[below]{$\rho$} coordinate (x axis);
\draw[->] (0,0) -- (0,3) node[left]{$D,g$} coordinate (y axis);
\draw (0,0) -- (0,-2);
\draw[thick] (0,1) .. controls (2,3) and (4,3) .. (6,0) node[above=3]{\footnotesize{$\alpha$}} node[midway, above]{$D$};
\draw[thick] (6,0) .. controls (7.1,-2) and (8,-2) .. (9,-1); 
\draw[dotted] (9,-1)--(9,0);
\filldraw[black] (6,0) circle (2pt);
\draw[thick,dashed] (0,0) .. controls (1,-2) and (1.9,-2) .. (3,0) node[above=3]{\footnotesize{$\gamma$}};
\draw[thick,dashed] (3,0) .. controls (5,3) and (7,3) .. (9,0) node[above]{\footnotesize{$1$}} node[midway,above]{$g$}; 
\filldraw[black] (3,0) circle (2pt);
\draw(5,-2.5) node[below]{$\alpha>\gamma$};

\begin{scope}[xshift=12cm]
\draw[->] (0,0) --  (10,0) node[below]{$\rho$} coordinate (x axis);
\draw[->] (0,0) -- (0,3) node[left]{$D,g$} coordinate (y axis);
\draw (0,0) -- (0,-2);
\draw[thick] (0,1) .. controls (2,3) and (4,3) .. (6,0) node[above=3, left=2]{\footnotesize{$\alpha=\gamma$}} node[midway, above]{$D$};
\draw[thick] (6,0) .. controls (7.1,-2) and (8,-2) .. (9,-1); 
\draw[dotted] (9,-1)--(9,0);
\draw[thick,dashed] (0,0) .. controls (2,-3) and (4,-3) .. (6,0) ;
\draw[thick,dashed] (6,0) .. controls (7.1,2) and (8,2) .. (9,0) node[above]{\footnotesize{$1$}} node[midway,above]{$g$}; 
\filldraw[black] (6,0) circle (2pt);
\draw(5,-2.5) node[below]{$\alpha=\gamma$};
\end{scope}

\begin{scope}[xshift=24cm]
\draw[->] (0,0) --  (10,0) node[below]{$\rho$} coordinate (x axis);
\draw[->] (0,0) -- (0,3) node[left]{$D,g$} coordinate (y axis);
\draw (0,0) -- (0,-2);
\draw[thick,dashed] (0,0) .. controls (2,-3) and (4,-3) .. (6,0) node[above=3]{\footnotesize{$\gamma$}} node[midway, above]{$g$};
\draw[thick,dashed] (6,0) .. controls (7.1,2) and (8,2) .. (9,0) node[above]{\footnotesize{$1$}}; 
\filldraw[black] (6,0) circle (2pt);
\draw[thick] (0,1) .. controls (1,2) and (1.9,2) .. (3,0) node[above=3]{\footnotesize{$\alpha$}};
\draw[thick] (3,0) .. controls (5,-3) and (7,-3) .. (9,-1) node[midway,above]{$D$}; \filldraw[black] (3,0) circle (2pt);
\draw[dotted] (9,-1)--(9,0);
\draw(5,-2.5) node[below]{$\alpha<\gamma$};
\end{scope}
\end{tikzpicture}
\end{center}
\caption{\label{f:fDpn-a>g}{The mutual behaviours of the functions $D$ (solid line) and $g$ (dashed line).}}
\end{figure}
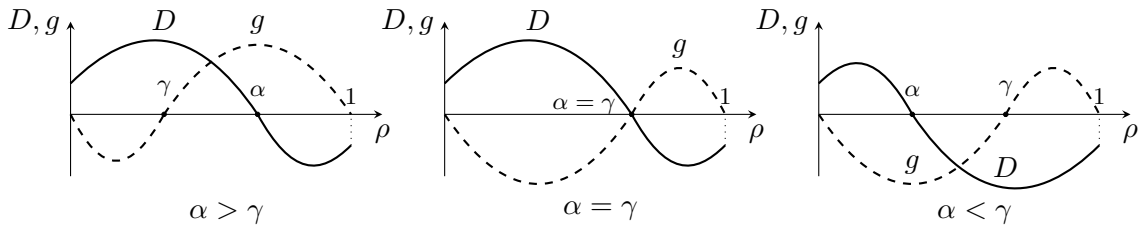

We focus on {\em wavefronts}, i.e., traveling-wave solutions whose profiles are globally-defined, noncostant and monotone. To fix ideas we deal with {\em non-increasing} profiles and this leads, because of (g), to require
\begin{equation}\label{e:infty}
\phi(-\infty)=1, \qquad \phi(\infty)=0.
\end{equation}
The study of non-decreasing profiles can be easily deduced.

\smallskip

The first results about solutions to problem \eqref{e:ODE}-\eqref{e:infty} in the bistable case were provided, in the case $f=0$ and $D=1$, in \cite{Aronson-Weinberger, Fife, Fife-McLeod, Kanel}: profiles were proved to exist only in correspondence of a unique value $c=c^*$, and the same result holds if $f$ is linear. The case when $f$ is nonlinear and $D=D(\rho)>0$ was studied in \cite{MMM2004, Tegner}, see also \cite[Theorem 10.36]{GK} when $D=D(\rho)\ge0$ and $f=0$. The case when $g$ may vanish more than three times was studied in \cite{Fife-McLeod, Kanel}. Further updated information about the bistable case can be found in \cite{Hamel}.The case when the diffusivity depends on $\rho_x$ is studied in \cite{Audrito}.

In the case $D$ changes sign, previous results were obtained in \cite{Maini-Malaguti-Marcelli-Matucci2007}, for the case when $f=0$ and $D$ satisfies (D), and \cite{Kuzmin-Ruggerini}, where $D$ changes sign twice (from the positive to the negative and again to the positive).
More precisely, in \cite{Maini-Malaguti-Marcelli-Matucci2007} the authors show that the conditions
\begin{equation}\label{e:gammaalfa}
\alpha>\gamma, \quad c>0,  \quad \hbox{ and } \quad \int_0^\alpha D(\rho)g(\rho)\, d\rho>0,
\end{equation}
are needed in order to have solutions. Under a further condition on the propagation speeds of the profiles connecting $\alpha$ to $0$ and $\alpha$ to $1$, they prove that there is a unique admissible speed $c^*$, which is strictly positive, and characterize the shapes of profiles according to the vanishings of $D$ in $0$ and $1$. The results in \cite{Kuzmin-Ruggerini} are analogous.

\smallskip


Our main results are contained in Theorem \ref{th:multiplicity}. We show that Equation \eqref{e:E} can have wavefronts for every $\alpha, \gamma \in (0,1)$, no matter the sign of $\alpha-\gamma$. If $\alpha>\gamma$ what we prove has still the flavor of \cite{Kuzmin-Ruggerini, Maini-Malaguti-Marcelli-Matucci2007}: profiles exist at most for a unique $c^*$ and the profile is uniquely determined up to space shifts (see Theorem \ref{th:multiplicity} $(a)$). If $\alpha\le \gamma$ our results are instead completely new. In this case, we show that if solutions exist, then this occurs for speeds ranging in a whole (bounded) interval (possibly degenerating to a single value). About the uniqueness, we prove what follows. In the case $\alpha=\gamma$, for every admissible speed $c$, the corresponding profile is here uniquely determined not only modulo space shifts but also modulo {\em plateaus} at level $\gamma$ (see Theorem \ref{th:multiplicity} $(b)$). Plateaus are stretches in $(0,1)$ where the profile is constant; they can be of arbitrary length. This phenomenon, which was already observed in an analogous situation (see \cite{CM-ZAMP}), is impossible if $\alpha>\gamma$. When $\alpha<\gamma$, the loss of uniqueness of the profiles associated to a same speed is way more dramatic. We show that if the interval of admissible speeds has a non-empty interior part, then to every speed therein corresponds a one-parameter family of (essentially different) profiles. Each of these solutions is  uniquely determined (up to space shifts) by the value of its derivative when the value of the profile itself reaches $\gamma$ (see Theorem \ref{th:multiplicity} $(c)$). Moreover, under assumptions on the steepness of $g$ at $\gamma$, we prove that to each admissible speed there correspond also profiles with plateaus at level $\gamma$. Theorem \ref{th:multiplicity} also shows that the derivative of a profile is strictly negative, with the possible exceptions of the points where $\phi$ reaches the values $0,\alpha,\gamma,1$.
%


We also obtain results that link proper convexity assumptions on $f$ to the occurrence of solutions. In particular, in Theorem \ref{thm:f convex} we show that if $f$ is convex then the behavior is still that observed in \cite{Kuzmin-Ruggerini, Maini-Malaguti-Marcelli-Matucci2007}: solutions exist only if \eqref{e:gammaalfa} holds but with $c>h(\alpha)$ replacing its second condition. Thus, in particular, profiles exist only if $\alpha>\gamma$. Moreover, in this case, we give a sufficient condition on the integral of the product $Dg$ to have  solutions.

In Theorem \ref{th:sufficient conditions}, we provide specific geometric sufficient conditions (again only on the terms $D$, $g$ and $f$) under which Equation \eqref{e:E} has wavefronts. Roughly speaking, these conditions are satisfied if $f$ is sufficiently strictly concave and $D,g$ satisfy some regularity conditions.


The main mathematical tool to study \eqref{e:ODE} with the boundary conditions \eqref{e:infty} is a reduction (in regions where $D$ has constant sign) of Equation \eqref{e:ODE} to singular first-order systems \cite{Fife}, see Section \ref{s:reduction1}. In order to obtain the results for profiles, our approach relies on refined upper- and lower-solutions based on comparison-type techniques  for the first-order problems.
\smallskip

Here follows the content of the paper. In Section \ref{sec:new} we give the basic definitions and state the main outcomes on profiles. Section \ref{s:reduction1} takes care of proving the link between profiles and singular first-order systems. In Section \ref{s:first order}, we give those technical results  for the first-order problems which are necessary to deal with our main focus. In Section \ref{s:proofs}, we mainly provide the proofs of the results of Section \ref{sec:new}, based on the tools in Sections \ref{s:reduction1} and \ref{s:first order}. In Section \ref{s:examples} we show some relevant explicit examples. Some results about the behavior of a profile at points where it reaches the value $\alpha$ are provided in Appendix \ref{s:appendix}.

%
%

\section{Main results}
\label{sec:new}
\setcounter{equation}{0}

Traveling-waves can fail to be of class $C^1$. The following definition makes precise what we mean by a TWs, see \cite{GK}.

\begin{definition}\label{d:tws} Assume $f,D,g\in C^0[0,1]$ and let $I\subset\R$ be an open interval. Let $\varphi\in C(I)$ be a function valued in $[0,1]$, which is differentiable a.e. and such that $D(\varphi) \varphi^{\, \prime}\in L_{\rm loc}^1(I)$; let $c$ be a real constant.

Then the function $\rho(x,t):=\varphi(x-ct)$, for $(x,t)$ with $x-ct \in I$, is a {\em traveling-wave} solution (briefly, a TW) to equation \eqref{e:E} with wave speed $c$ and wave profile $\phi$ if we have
\begin{equation}\label{e:def-tw}
\int_I \left(D\left(\phi(\xi)\right)\phi'(\xi) - f\left(\phi(\xi)\right) + c\phi(\xi) \right)\psi'(\xi) - g\left(\phi(\xi)\right)\psi(\xi)\,d\xi =0,
\end{equation}
for every $\psi\in C_0^\infty(I)$.
\end{definition}

We say that a TW is {\em global} if $I=\R$, while it is {\em strict} if $I\ne \R$ and $\phi$ cannot be extended to $\R$; a TW is {\em classical} if $\varphi$ is differentiable, $D(\varphi) \varphi'$ is absolutely continuous and \eqref{e:ODE} holds a.e.; a TW is {\em sharp at $\ell$} if there exists $\xi_{\ell}\in I$, with $\phi(\xi_{\ell})=\ell$, such that $\phi$ is classical in $I\setminus\{\xi_\ell\}$ and not differentiable at $\xi_{\ell}$. Analogously, a TW is {\em classical} at $\ell$ if it is classical in a neighborhood of $\xi_\ell$.

A TW is {\em a wavefront} if it is global, with a monotone, non-constant profile $\phi$ which satisfies either \eqref{e:infty} or the converse condition; a speed $c$ is {\em admissible} if there exists a wavefront with speed $c$. A TW is also called: {\em a semi-wavefront to} $1$ (or {\em to} $0$) if $I=(a,\infty)$ for $a\in\R$, the profile $\phi$ is monotonic, non-constant and $\phi(\xi)\to1$ (respectively, $\phi(\xi)\to0$) as $\xi\to\infty$; {\em a semi-wavefront from} $1$ (or {\em from} $0$) if $I=(-\infty,b)$ for $b\in\R$, the profile $\phi$ is monotonic, non-constant and $\phi(\xi)\to1$ (respectively, $\phi(\xi)\to0$) as $\xi\to-\infty$. About semi-wavefronts, we say that $\phi$ connects $\phi(a^+)$ ($1$ or $0$) with $1$ or $0$ (resp., with $\phi(b^-)$).

The problem of the loss of regularity of $\phi$ depends on whether the parabolic equation degenerates or not; more precisely, by arguing on equation \eqref{e:ODE}, it is easy to see that if $f, D$ are of class $C^1$ and $g$ of class $C^0$, then $\phi$ is classical in every interval $I_{\pm}\subseteq I$ where $\pm D\left(\phi(\xi)\right) > 0$ for $\xi \in I_{\pm}$; moreover, $\phi\in C^2(I_\pm)$ (see e.g. \cite[Lemma 2.20]{GK}).


We denote the {\em difference quotient} of a function $F=F(\phi)$ with respect to a point $\phi_0$ as
\[
\delta(F,\phi_0)(\phi) := \frac{F(\phi)-F(\phi_0)}{\phi-\phi_0}.
\]

%
%
%

\smallskip

In the following we need some growth conditions on $g$ at $\gamma$ that we state now. More precisely, on a case by case basis, we assume that there exists $L>0$ such that either
\begin{equation}
\label{g sublinear gamma}
|g(\phi)| \le L|\phi-\gamma|\ \mbox{ in a left- or right-neighborhood of $\gamma$,}
\end{equation}
or
 \begin{equation}
 \label{integrability g}
 |g(\phi)| \ge L\left|\phi-\gamma\right|^\tau,\ \tau\in(0,1), \ \mbox{in a full neighborhood of $\gamma$.
}
 \end{equation}

The following example describes a very simple case in which the presence in \eqref{e:E} of a non-zero $f$ allows the existence of wave profiles regardless of the relative order between $\alpha$ and $\gamma$. We highlight here the first important difference with the case $f=0$, in which we have solutions only if $\alpha>\gamma$ (see \eqref{e:gammaalfa}).

\begin{example}
\label{example1}
{
\rm
For $\alpha, \gamma \in (0,1)$, define, for $\phi \in [0,1]$,
\[ D(\phi)=\left(\alpha-\phi\right), \ g(\phi)=\phi\left(1-\phi\right)\left(\phi-\gamma\right).\]
Hence, $D$ satisfies (D) and $g$ satisfies (g). Let $f$ be defined by $f(\phi)=-\phi^3+\frac{2\alpha +1}{2}\phi^2+(\gamma-\alpha)\phi$; hence, $h(\phi)=-3\phi^2+(2\alpha+1)\phi+(\gamma-\alpha)$. A direct check shows that $\phi=\phi(\xi) \in (0,1)$, defined  as
\begin{equation}
\label{e:ex1phi}
\phi(\xi):=\frac1{e^\xi+1}, \ \xi \in \R,
\end{equation}
satisfies \eqref{e:ODE}  with $c=0$ everywhere in $\R$. Also, $\phi$ is strictly decreasing, $\phi(-\infty)=1$ and $\phi(\infty)=0$. Hence, $\phi$ is the profile of a wavefront with wave speed $c=0$.
}
\end{example}

Besides the very simple Example \ref{example1}, we shall explore a large variety of possible profiles, due to a non-zero $f$. Among new families of profiles, there are those which deal with {\em stretchings at level $\gamma$}. A wave profile $\phi$ with speed $c$ is said to be unique up to stretchings at level $\gamma$ if, for each $\delta_1,\delta_2>0$,  the functions $\phi_{\delta_1,\delta_2}$ defined by
\begin{equation}
\label{stretch}
\phi_{\delta_1,\delta_2}(\xi):=\begin{cases} \phi(\xi+\delta_1) \ &\mbox{ if } \ \xi < \xi_\gamma -\delta_1,\\
\gamma \ &\mbox{ if } \ \xi_\gamma-\delta_1 < \xi < \xi_\gamma+\delta_2,\\
\phi(\xi-\delta_2) \ &\mbox{ if }  \ \xi>\xi_\gamma+\delta_2,
\end{cases}
\end{equation}
where $\xi_\gamma$ is a point such that $\phi(\xi_\gamma)=\gamma$, are still wave profiles with speed $c$; see Figure \ref{f:plateaus-slopes} on the left. Note, this kind of solutions appeared also in \cite{CM-ZAMP}; they are possible because $\gamma$ is a zero of $g$ and, roughly speaking, if $\phi$ is such that $D(\phi)\phi\rq{}$ vanishes when $\phi$ reaches $\gamma$ (see Lemma \ref{lem:def}). If $\alpha>\gamma$ plateaus never occur, while they always do if $\alpha=\gamma$.  The case $\alpha<\gamma$ is more controversial: if \eqref{g sublinear gamma} is satisfied, then plateaus do not occur; if it fails, then we show in Examples \ref{other} and \ref{other2}, that profiles {\em can} have plateaus or do {\em not}.

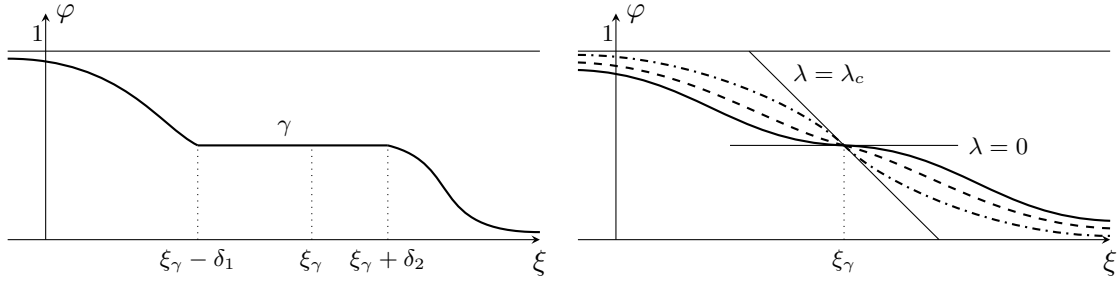
\begin{figure}[htb]
\begin{center}

\begin{tikzpicture}[>=stealth, scale=0.5]
\draw[->] (-7,0) --  (7,0) node[below]{$\xi$} coordinate (x axis);
\draw[->] (-6,0) -- (-6,6) node[right]{$\phi$} coordinate (y axis);
\draw (-6,5) node[left=3, above]{\footnotesize$1$} -- (7,5);
\draw (-7,5) -- (-6,5);
\draw[thick] (-7,4.8) .. controls (-4,4.8) and (-3,3) .. (-2,2.5);
\draw[dotted] (-2,0) node[below]{\footnotesize$\xi_\gamma-\delta_1$} -- (-2,2.5);
\draw[dotted] (1,0) node[below]{\footnotesize$\xi_\gamma$} -- (1,2.5);
\draw[thick] (-2,2.5) -- (3,2.5);
\draw[dotted] (3,0) node[below]{\footnotesize$\xi_\gamma+\delta_2$} -- (3,2.5);
\draw (0,2.5) node[right=4, above]{\footnotesize$\gamma$};
\draw[thick] (3,2.5) .. controls (5,2) and (4,0.2) .. (7,0.2);

\begin{scope}[xshift=15cm]
\draw[->] (-7,0) --  (7,0) node[below]{$\xi$} coordinate (x axis);
\draw[->] (-6,0) -- (-6,6) node[right]{$\phi$} coordinate (y axis);
\draw (-6,5) node[left=3, above]{\footnotesize$1$} -- (7,5);
\draw (-7,5) -- (-6,5);
\draw[thick] (-7,4.5) .. controls (-4,4.4) and (-3,2.6) .. (0,2.5);
\draw[dotted] (0,0) node[below]{\footnotesize{$\xi_\gamma$}} -- (0,2.5);
\draw[thick] (0,2.5) .. controls (3,2.4) and (4,0.6) .. (7,0.5);

\draw[thick, dashed] (-7,4.7) .. controls (-4,4.6) and (-2,3) .. (0,2.5);
\draw[thick,dashed] (0,2.5) .. controls (2,2) and (4,0.4) .. (7,0.3);

\draw[thick, dashdotted] (-7,4.9) .. controls (-4,4.9) and (-1,3.5) .. (0,2.5);
\draw[thick,dashdotted] (0,2.5) .. controls (1,1.5) and (4,0.1) .. (7,0.1);

\draw (-3,2.5) -- (3,2.5) node[right]{\footnotesize{$\lambda=0$}};
\draw (2.5,0) -- (-2.5,5) node[right=30, below=1]{\footnotesize{$\lambda=\lambda_c$}};
\end{scope}

\end{tikzpicture}
\end{center}
\caption{\label{f:plateaus-slopes}{On the left: a profile with a plateau at level $\gamma$ from $\xi_\gamma-\delta_1$ to $\xi_\gamma+\delta_2$. On the right: some profiles in case {\em (c)} in Theorem \ref{th:multiplicity}}; profiles have been shifted so that $\phi(\xi_\gamma)=\gamma$.}
\end{figure}

We now state our main results. We chose to keep the statements as simple as possible, favoring immediacy versus completeness; much more detailed (and, unfortunately, technical) results are either provided in the next sections or can be easily deduced as outlined below.

The first result shows the existence of new families of solutions with respect to the case $f=0$, which concern not only the existence of profiles in the cases $\alpha\le\gamma$ but also the {\em intrinsic} loss of uniqueness. The latter is due either to the formation of plateaus (in the case $\alpha=\gamma$ and $\alpha<\gamma$ if $\lambda=0$), or to the existence of substantially different solutions (in the case $\alpha<\gamma$, for $\lambda<0$). We recall that a speed $c$ is admissible for Equation \eqref{e:E} if there exists a wavefront with speed $c$.

\begin{theorem}
\label{th:multiplicity}
Assume (f), (g) and (D), and let $\mathcal{J}$ be the set of admissible speeds for Equation \eqref{e:E}. Then either $\mathcal{J}$ is empty or
\begin{enumerate}[(a)]
\item if $\alpha > \gamma$ then $\mathcal{J}$ contains a unique speed, and the corresponding profile is unique (up to shifts);

\item if $\alpha=\gamma$ then $\mathcal{J}$ is a bounded interval, and for every $c\in\mathcal{J}$ the corresponding profile is unique (up to shifts and stretchings at level $\gamma=\alpha$). Moreover, if \eqref{integrability g} holds then $\mathcal J$ must be also closed;

\item if $\alpha < \gamma$ then $\mathcal{J}$ is a bounded interval. For every $c \in \mathrm{int}\,\mathcal J$ there is a family of profiles $\{\phi_\lambda\}_\lambda$ determined by
\[\phi_\lambda\rq{}(\xi_\gamma) =\lambda\in [\lambda_c,0),\]
for some $\lambda_c<0$, where $\xi_\gamma$ satisfies (uniquely) $\phi_\lambda(\xi_\gamma)=\gamma$ and $\phi_\lambda$ is unique up to shifts.

If \eqref{integrability g} holds, then $\mathcal J$ is closed and, for every $c\in \mathcal J$, there exist profiles corresponding to $\lambda=0$. In such a case, $\phi_0$ is unique up to shifts and stretchings at level $\gamma$.
\end{enumerate}
Moreover, suppose that a profile $\phi$ exists. If $\alpha \ge \gamma$ then  $\phi\rq{}<0$ if $\phi \neq 0,\alpha,1$ while if  $\alpha<\gamma$  then $\phi\rq{}<0$ if $\phi \neq 0,\alpha, \gamma,1$; in case $\alpha<\gamma$ and if \eqref{g sublinear gamma} holds, then $\phi\rq{}<0$ also if $\phi=\gamma$.
\end{theorem}


We now comment on Theorem \ref{th:multiplicity}. First, when $\alpha \neq \gamma$, every admissible profile $\phi$ must satisfy $\phi\rq{}(\xi_\alpha)<0$ where $\phi(\xi_\alpha)=\alpha$, because otherwise \eqref{e:ODE} would imply that $\alpha$ must be a zero of $g$. Second, in $(b)$ and $(c)$ of Theorem \ref{th:multiplicity} the interval $\mathcal J$ can degenerate to a single value, as we show in Example \ref{ex:Jeqzero}. Third, there is no \lq\lq continuity\rq\rq\ of results when $\gamma\to0$: if $\alpha>\gamma>0$ we have at most one admissible speed, while in the case $\gamma=0$ there is a whole half line of such speeds \cite{BCM2}. We refer to Figure \ref{f:plateaus-slopes} on the right for the case {\em (c)}.

\smallskip

The next theorem generalizes a result proved in \cite[Theorem 3.1]{Maini-Malaguti-Marcelli-Matucci2007} in the case $f=0$, see \eqref{e:gammaalfa}, and provides necessary and sufficient conditions for the existence of profiles when $f$ is convex.

\begin{theorem}
\label{thm:f convex}
Under (f), (g) and (D), assume that $f$ is convex.
\par
If Equation \eqref{e:E} has wavefronts then
\begin{equation}\label{e:fconv1} \int_{0}^{\alpha} D(\rho)g(\rho)\,d\rho> 0.\end{equation}

If it holds that
\begin{equation}\label{fconvex suff cond} \int_{0}^{\alpha}D(\rho)g(\rho)\,d\rho \ge \frac{2\alpha^2\Sigma}{3}  \left(\Sigma+\sqrt{\Sigma^2+\frac2{\alpha} M}\right),\end{equation}
where $\Sigma:=h(1)-h(0)+2\sqrt{\sup_{[\alpha,1)}\delta\left(Dg,1\right)}$ and $M:=\max_{\rho\in [0,\gamma]}-D(\rho)g(\rho)>0$,
then Equation \eqref{e:E} has wavefronts.
\end{theorem}

Some comments on conditions \eqref{e:fconv1} and \eqref{fconvex suff cond} now follow. A consequence of \eqref{e:fconv1} is that $\alpha>\gamma$; in this case we deduce, from the proof of Theorem \ref{thm:f convex}, that $c>h(\alpha)$ if $f$ is convex (see \eqref{e:simple estimate}). In particular, Theorem \ref{thm:f convex} when $f=0$ gives \eqref{e:gammaalfa}. Condition \eqref{fconvex suff cond} is the extension of \cite[(3.13)]{Maini-Malaguti-Marcelli-Matucci2007} where, there, $\sigma$ is to be read as $\Sigma^2/4$. Since \eqref{fconvex suff cond} implies $\alpha>\gamma$, we notice that the right-hand side in \eqref{fconvex suff cond} involves only regions where the product $Dg$ is negative; also \eqref{e:fconv1} states that the region where $Dg$ is positive must prevail on the region where $Dg$ is negative. Condition \eqref{e:fconv1} does not depend on $f$, on the contrary of \eqref{fconvex suff cond}, but this is a mere consequence of the assumption of convexity. In the general case, the corresponding condition does involve $f$.

\smallskip

Our last main result concerns instead some sufficient conditions for the existence of the profiles in the case $f$ is strictly concave.

\begin{theorem}
\label{th:sufficient conditions}
Assume (f), (g) and (D). Also, suppose that $f$ is strictly concave. Then, Equation \eqref{e:E} admits wavefronts if
\begin{enumerate}[(a)]
\item $\alpha>\gamma$, $\dot g(\gamma)$ exists, it is positive and
\begin{equation}
\label{sufficient condition1}
\dot f(\gamma)-\frac{f(1)-f(\alpha)}{1-\alpha} \ge 2\left(\sqrt{\sup_{[\alpha,1)}\delta\left(Dg,1\right)}+\sqrt{\sup_{[0,\gamma)}\delta\left(Dg,\gamma\right)}\right);
\end{equation}
\item $\alpha=\gamma$, \eqref{integrability g} and \eqref{sufficient condition1};
\item $\alpha < \gamma$ and
\begin{equation}
\label{sufficient condition2}
\dot f(\gamma) - \frac{f(1)-f(\gamma)}{1-\gamma} > 2 \left(\sqrt{\sup_{[\gamma,1)}\delta\left(Dg,1\right)}+\sqrt{\sup_{[0,\gamma]\setminus\{\alpha\}}\delta\left(Dg, \alpha\right)}\right).
\end{equation}
\end{enumerate}
\end{theorem}

We now comment on conditions \eqref{sufficient condition1} and \eqref{sufficient condition2}. In order to fulfill \eqref{sufficient condition1}, a necessary assumption is that $Dg$ has finite slope at $\gamma^-$. Thus, in Part $(a)$ the condition \eqref{integrability g} is implicitly excluded (because $D(\gamma)\neq 0$). Instead, \eqref{integrability g} and \eqref{sufficient condition1} are not in contradiction in Part $(b)$ since, in this case, $Dg$ has derivative (which is zero) at $\gamma$.

\smallskip

We conclude this section by briefly giving an information on the behavior of a profile $\phi$ where $D$ degenerates.  In Appendix \ref{s:appendix} we compute $\phi\rq{}$ at points where $\phi=\alpha$, by exploiting precise results from \cite{BCM1, BCM2, CdRM}. We show the occurrence (according to the mutual order of $\alpha$ and $\gamma$, and some implicit conditions on the wave speed) of sharp profiles at $\alpha$, precisely with infinite slope at $\alpha$ or with different right- and left-derivatives at $\alpha$. This contrasts with the case $f=0$, where profiles are classical in $\alpha$. The behavior of $\phi$ at the equilibria $0$ and $1$ can be discussed as follows. Because of \eqref{e:mathcalVzero}, if $D(0)>0$ or $D(1)<0$ then $\phi$ is classical at $0$ or at $1$, respectively. If $D(0)=0$ or $D(1)=0$, $\phi$ can be sharp at $0$ or $1$, respectively. The regularity of the wavefronts approaching $0$ and $1$, as well as information on their strict monotonicity, can be deduced by \cite[Theorem 2.3]{CdRM} and \cite[Theorem 2.3]{BCM1} (see also \cite{MMconv}), after straightforward manipulations.


\section{Reduction to a singular first-order problem}
\label{s:reduction1}
\setcounter{equation}{0}

In this section we take advantage of the monotonicity of the profiles to reduce the solvability of the problem for $\phi$, given by \eqref{e:ODE} and \eqref{e:infty}, to the solvability of proper first-order singular problems. This reduction shall be the main tool for the proofs of our results.

 We begin by giving a result that is instrumental in what follows.

\begin{lemma}
\label{lem:def}
Assume that $\phi$ is a profile of a wavefront with speed $c$, as in Definition \ref{d:tws}. Then, the function
\begin{equation}
\label{e:mathcalV}
\mathcal{V}(\xi):=D\left(\phi(\xi)\right) \phi\rq{}\left(\xi\right), \ \mbox{ for a.e. } \ \xi \in \R,
\end{equation}
has a continuous extension to the whole $\R$. Moreover, it holds:
\begin{equation}
\label{e:mathcalVzero}
\lim_{\xi \to \xi_1^+} \mathcal{V}(\xi)=0 \ \mbox{ and } \ \lim_{\xi \to \xi_0^-} \mathcal{V}(\xi)=0,
\end{equation}
where $\xi_1:=\inf\{\xi: \phi(\xi)<1\}$ and $\xi_0:=\sup\{\xi: \phi(\xi)>0\}$.
\end{lemma}

\begin{proof}
 Where $D(\phi) \neq 0$, solutions of \eqref{e:ODE} (in the sense of Definition \ref{d:tws}) are regular (see \cite[Lemma 2.20]{GK}). Because of this and since $\phi$ is supposed to be monotone (decreasing), then for every $\zeta \in \R$ such that $D\left(\phi (\zeta)\right)=0$, there always exists a neighborhood of $\zeta$ (say $(\zeta-\delta, \zeta +\delta)$, for some $\delta>0$) such that $\phi$ (restricted to $(\zeta-\delta, \zeta +\delta)$) may be not differentiable only at $\zeta$. With this in mind, take an arbitrary test function $\psi \in C_0^\infty(\R)$ with support in $(\zeta-\delta, \zeta +\delta)$. By separating the contributions in $\int_{\zeta-\delta}^{\zeta}$ and $\int_{\zeta}^{\zeta+\delta}$ and by applying integration by parts, we prove that
 \begin{equation}
  \label{e:continuityV}
 0=
  \int_{\zeta-\delta}^{\zeta+\delta} \left(D(\phi)\phi\rq{}+c\phi-f(\phi)\right)\psi\rq{}-g(\phi)\psi\,d\xi =  \left(D(\phi)\phi\rq{}\right)(\zeta^+)\psi(\zeta)-  \left(D(\phi)\phi\rq{}\right)(\zeta^-)\psi(\zeta).
 \end{equation}

It remains to prove \eqref{e:mathcalVzero}. We show $\eqref{e:mathcalVzero}_1$, the other runs similarly. If $\xi_1 \in \R$, then \eqref{e:mathcalVzero} follows from \eqref{e:continuityV}. If $\xi_1=-\infty$, by integrating \eqref{e:ODE} and since $g$ has constant sign near $1$ we deduce that $\mathcal{V}$ has limit for $\xi \to - \infty$. The boundedness of $\phi$ implies then $\eqref{e:mathcalVzero}_1$.
\end{proof}

In the case $\alpha \ge \gamma$, the existence of profiles implies the existence of solutions to the first-order problem
\begin{equation}
\label{e:problem0}
\begin{cases}
\dot{z}=h-c- q/z \ &\mbox{ in } \ (0,\alpha)\cup (\alpha,1),\\
z<0 \ &\mbox{ in } \ (0,\alpha),\\
z>0 \ &\mbox{ in } \ (\alpha,1),\\
z(0)=z(\alpha)=z(1)=0.
\end{cases}
\end{equation}
Solutions to \eqref{e:problem0} are sought in the class $C^0\left[0,1\right]\cap C^1\left(\left(0,1\right)\setminus\{\alpha\}\right) $. In the case $\alpha<\gamma$ the reduction depends also on the behaviour of $g$ near $\gamma$. If $g$ is regular enough then the existence of the profiles implies the solvability of \eqref{e:problem0}, again. Otherwise, the existence of the profiles imply only the solvability of the following problem
\begin{equation}
\label{e:problem0prime}
\begin{cases}
\dot{z}=h-c- q/z \ &\mbox{ in } \ (0,\alpha)\cup (\alpha,\gamma)\cup (\gamma,1),\\
z<0 \ &\mbox{ in } \ (0,\alpha),\\
z>0 \ &\mbox{ in } \ (\alpha,\gamma)\cup (\gamma,1),\\
z(0)=z(\alpha)=z(1)=0.
\end{cases}
\end{equation}
where $z \in  C^0[0,1] \cap C^1\left(\left(0,1\right)\setminus\{\alpha,\gamma\}\right)$. Note, Problem \eqref{e:problem0} is \eqref{e:problem0prime} with the additional requirement $z(\gamma)\neq 0$.

\begin{proposition}
\label{prop:nec}
Assume that there exists a wavefront of \eqref{e:E} with profile $\phi$ satisfying \eqref{e:infty}, with speed $c$. Then,  $\phi\rq{}(\xi)<0$ for every $\xi \in\R$ such that $\phi(\xi)\neq 0,1, \alpha, \gamma$. Moreover, we have:
\begin{enumerate}[(i)]
\item
When $\alpha \ge \gamma$, there exists a solution of \eqref{e:problem0}.
\item
When $\alpha<\gamma$, there exists a solution of \eqref{e:problem0prime}; if \eqref{g sublinear gamma} holds, then $z(\gamma)>0$ and hence $z$ satisfies \eqref{e:problem0}.
\end{enumerate}
\par\noindent
Moreover, when either $\alpha > \gamma$ or both $\alpha <\gamma$ and \eqref{g sublinear gamma} hold, then there exists a unique $\xi_\gamma\in\R$ such that $\phi(\xi_\gamma)=\gamma$, and it holds $\phi\rq{}(\xi_\gamma)<0$.
\end{proposition}


\begin{proof}
Assume that \eqref{e:E} admits a wavefront with profile $\phi$ satisfying \eqref{e:ODE}-\eqref{e:infty} and wave speed $c\in \R$. Since $\phi$ is assumed to be monotone decreasing and $\phi$ is a classical solution of \eqref{e:ODE} where $D\neq 0$ (see \cite[Lemma 2.20]{GK}), we have that $\phi\rq{}\le 0$ if  $\phi \notin \{0,1,\alpha\}$. We show that necessarily it holds
\begin{equation}\label{e:Messi1}
\phi\rq{}< 0 \ \mbox{ if } \ \phi \notin \{0,1,\alpha,\gamma\}.
\end{equation}
 Suppose, by contradiction, that  $\phi\rq{}(\xi^*_0)=0$ for some $\xi^*_0 \in\R$ with $\phi_0^*:=\phi(\xi_0^*)\notin \{0,1,\alpha,\gamma\}$. We consider the case $D(\phi_0^*)>0$ and $g(\phi_0^*)>0$; the other cases run similarly. Since $D\left(\phi_0^*\right)>0$, we have that $\phi$ is classical in a neighborhood of $\xi_0^*$. Hence $\mathcal{V}$ (as in \eqref{e:mathcalV}) belongs to $C^1$, in a neighborhood of $\xi_0^*$, and satisfies
\[
\mathcal{V}(\xi^*_0)=0 \ \mbox{ and }  \ {\mathcal{V}}\rq{}(\xi_0^*)=-g(\phi_0^*)<0.
\]
Hence, $\mathcal{V}>0$ in $(\xi_0^*-\delta, \xi_0^*)$, for some $\delta>0$ small enough to also have $D(\phi(\xi))>0$, for every $\xi\in (\xi_0^*-\delta, \xi_0^*)$. Hence, $\phi\rq{}>0$ in $(\xi_0^*-\delta, \xi_0^*)$. This is in contradiction with $\phi\rq{}\le 0$ if  $\phi \notin \{0,1,\alpha\}$  and then proves \eqref{e:Messi1}.

We separate the cases $\alpha >\gamma$, $\alpha=\gamma$ and $\alpha<\gamma$ and refer to Figure \ref{f:phimanyxi}

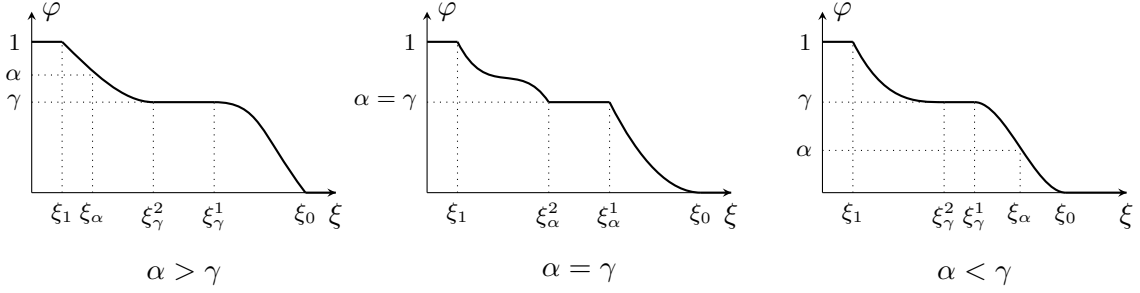
\begin{figure}[htb]
\begin{center}

\begin{tikzpicture}[>=stealth, scale=0.4]
\draw[->] (0,0) --  (1,0) node[below]{$\xi$} coordinate (x axis);
\draw (-9,0) -- (0,0);
\draw[->] (-9,0) -- (-9,6) node[right]{$\phi$} coordinate (y axis);
%
\draw[thick] (-9,5) node[left]{\footnotesize$1$}-- (-8,5);
\draw[dotted] (-9,3.9) node[left]{\footnotesize$\alpha$} -- (-7,3.9);
\draw[dotted] (-7,0) node[below]{\footnotesize$\xi_\alpha$} -- (-7,3.9);
\draw[dotted] (-8,0) node[below]{\footnotesize$\xi_1$} -- (-8,5);
\draw[thick] (-8,5) .. controls (-7,4) and (-6,3) .. (-5,3);
\draw[dotted] (-9,3) node[left]{\footnotesize$\gamma$} -- (-5,3);
\draw[thick] (-5,3) -- (-3,3);
\draw[dotted] (-5,0) node[below]{\footnotesize$\xi_\gamma^2$} -- (-5,3);
\draw[dotted] (-3,0) node[below]{\footnotesize$\xi_\gamma^1$} -- (-3,3);
\draw[thick] (-3,3) .. controls (-1.5,3) and (-1.5,2) .. (0,0);
\draw[thick] (0,0) node[below]{\footnotesize$\xi_0$}-- (1,0);
\draw (-4,-2) node[below]{$\alpha>\gamma$};

\begin{scope}[xshift=13cm]
\draw[->] (0,0) --  (1,0) node[below]{$\xi$} coordinate (x axis);
\draw (-9,0) -- (0,0);
\draw[->] (-9,0) -- (-9,6) node[right]{$\phi$} coordinate (y axis);
\draw[thick] (-9,5) node[left]{\footnotesize$1$}-- (-8,5);
\draw[dotted] (-8,0) node[below]{\footnotesize$\xi_1$} -- (-8,5);
\draw[thick] (-8,5) .. controls (-7,3) and (-6,4.5) .. (-5,3);
\draw[dotted] (-9,3) node[left]{\footnotesize$\alpha=\gamma$} -- (-5,3);
\draw[thick] (-5,3) -- (-3,3);
\draw[dotted] (-5,0) node[below]{\footnotesize$\xi_\alpha^2$} -- (-5,3);
\draw[dotted] (-3,0) node[below]{\footnotesize$\xi_\alpha^1$} -- (-3,3);
\draw[thick] (-3,3) .. controls (-2,1) and (-1,0) .. (0,0);
\draw[thick] (0,0) node[below]{\footnotesize$\xi_0$}-- (1,0);
\draw (-4,-2) node[below]{$\alpha=\gamma$};
\end{scope}

\begin{scope}[xshift=26cm]
\draw[->] (0,0) --  (1,0) node[below]{$\xi$} coordinate (x axis);
\draw (-9,0) -- (0,0);
\draw[->] (-9,0) -- (-9,6) node[right]{$\phi$} coordinate (y axis);
\draw[thick] (-9,5) node[left]{\footnotesize$1$}-- (-8,5);
\draw[dotted] (-8,0) node[below]{\footnotesize$\xi_1$} -- (-8,5);
\draw[thick] (-8,5) .. controls (-7,3) and (-6,3) .. (-5,3);
\draw[dotted] (-9,3) node[left]{\footnotesize$\gamma$} -- (-5,3);
\draw[thick] (-5,3) -- (-4,3);
\draw[dotted] (-5,0) node[below]{\footnotesize$\xi_\gamma^2$} -- (-5,3);
\draw[dotted] (-4,0) node[below]{\footnotesize$\xi_\gamma^1$} -- (-4,3);
\draw[thick] (-4,3) .. controls (-3,3) and (-2,0) .. (-1,0);
\draw[dotted] (-9,1.4) node[left]{\footnotesize$\alpha$}-- (-2.5,1.4);
\draw[dotted] (-2.5,0) node[below]{\footnotesize$\xi_\alpha$} -- (-2.5,1.5);
\draw[thick] (-1,0) node[below]{\footnotesize$\xi_0$}-- (1,0);
\draw (-4,-2) node[below]{$\alpha<\gamma$};
\end{scope}
\end{tikzpicture}
\end{center}
\caption{\label{f:phimanyxi}{The profiles. In the case $\alpha>\gamma$, we show that $\xi_\gamma^2=\xi_\gamma^1$. In the case $\alpha<\gamma$, we show $\xi_\gamma^2=\xi_\gamma^1$ under \eqref{g sublinear gamma}.}}
\end{figure}

{\em Assume $\alpha >\gamma$}. From \eqref{e:Messi1} we deduce that $\phi$ is invertible in each of the three intervals given by $\phi \in (0,\gamma)$, $(\gamma,\alpha)$ and $(\alpha,1)$. We define then
\begin{gather*}
\xi_1:=\inf\{\xi: \phi(\xi)<1\},\quad \xi_0:=\sup\{\xi: \phi(\xi)>0\},\\
\xi_\alpha:=\mbox{the unique $\xi$ such that $\phi(\xi)=\alpha$},\\
\xi_\gamma^2:=\sup\{\xi: \phi(\xi)>\gamma\},\quad \xi_\gamma^1:=\inf\{\xi: \phi(\xi)<\gamma\},
\end{gather*}
so that $-\infty \le \xi_1 < \xi_\alpha< \xi_\gamma^2 \le \xi_\gamma^1  < \xi_0\le \infty$.

%

Observe, $\xi_\alpha$ is unique since otherwise $\phi$ must be constantly equal to $\alpha$ in an interval and this contradicts \eqref{e:ODE} because of $g(\alpha)\neq 0$. We can define three functions $\zeta_1:(0,\gamma)\to (\xi_\gamma^1,\xi_0)$, $\zeta_2:(\gamma,\alpha)\to (\xi_\alpha, \xi_\gamma^2)$ and $\zeta_3:(\alpha,1)\to (\xi_1, \xi_\alpha)$, such that $\zeta_i=\phi^{-1}$ in each corresponding interval. Also, we define
\begin{equation}
\label{e:zeta1}
z_1(\phi):=D(\phi)\phi\rq{}\left(\zeta_1(\phi)\right), \ \phi \in (0,\gamma),
\end{equation}
and, analogously, $z_2(\phi):= D(\phi) \phi\rq{}\left(\zeta_2(\phi)\right)$, for $\phi \in (\gamma, \alpha)$, and $z_3(\phi):=D(\phi)\phi\rq{}\left(\zeta_3(\phi)\right)$, for $\phi \in (\alpha,1)$.
It follows from their definitions that $z_1,z_2,z_3$ are of class $C^1$ and (from \eqref{e:ODE}) that they solve \begin{equation}
\label{e:dot z i}
 \dot{z_i}(\phi) =h(\phi)-c-\frac{D(\phi)g(\phi)}{z_i(\phi)}, \ \mbox{ for } \ i=1,2,3,
 \end{equation}
where they are defined.

Since $\mathcal{V}$ in \eqref{e:mathcalV} is continuous in $\R$, we have $z_1(\gamma)=z_2(\gamma)$, $z_2(\alpha)=z_3(\alpha)$, and we can define a function $z=z(\phi)$, for $\phi \in [0,1]$, by gluing together the continuous extensions of $z_1,z_2,z_3$. Moreover, $z \in C^0[0,1]$ and $z \in C^1$ in $(0,\gamma)\cup (\gamma,\alpha)\cup (\alpha,1)$.

Since $\mathcal{V}$ satisfies \eqref{e:mathcalVzero}, then $z(0)=z(1)=0$. By \eqref{e:Messi1} we have $z<0$ in $(0,\gamma)\cup(\gamma,\alpha)$ and $z>0$ in $(\alpha,1)$. Moreover, $z(\alpha)=0$. This proves that $z$ satisfies \eqref{e:problem0prime}.  To show that $z$ is a solution of \eqref{e:problem0} it remains to prove that $z_1(\gamma)<0$ and that $z$ is of class $C^1$ at $\gamma$. To this end, suppose by contradiction that $z_1(\gamma)=0$. Since $\dot{z_1}=h-c-Dg/z_1$, and $Dg <0$ in $(0,\gamma)$, we have
\[
z_1(\gamma)-z_1(\phi) =\int_{\phi}^{\gamma}\dot{z_1}(\sigma)\,d\sigma<\int_{\phi}^{\gamma}\left[h(\sigma)-c\right]\,d\sigma \le \bigl(\max_{[0,\gamma]} h-c\bigr)\left(\gamma-\phi\right), \quad \phi\in(0,\gamma).
\]
 Thus, with $\Delta:=\max_{[0,\gamma]} h-c$,
 \begin{equation}
 \label{z1}
 z_1(\phi) > \Delta \left(\phi-\gamma\right), \ \mbox{ for } \ \phi < \gamma.
 \end{equation}
 Observe that $\Delta>0$, since $\Delta > \frac{-z_1(\phi)}{\gamma-\phi}>0$, for every $\phi \in (0,\gamma)$.
 From \eqref{z1}, $D(\gamma) >0 $ and
 \[
\zeta_1(\gamma) - \zeta_1(\delta)=\int_{\delta}^{\gamma} \dot{\zeta_1}(\sigma)\,d\sigma= \int_{\delta}^{\gamma} \frac{D(\sigma)}{z_1(\sigma)}\,d\sigma,
\]
for $\delta\in(0,\gamma)$, we obtain, for $\delta$ close to $\gamma$,
\begin{equation}\label{e:Pirlo}
\xi_\gamma^1=\zeta_1(\gamma) < \zeta_1(\delta) + \int_{\delta}^{\gamma} \frac{D(\gamma)+o(1)}{\Delta \left(\sigma -\gamma\right)}\,d\sigma=-\infty.
\end{equation}
This contradicts $\xi_\gamma^1 > \xi_\alpha$. Then, we proved that $z_1(\gamma)<0$ (and so $z(\gamma)$). Moreover, observe that $z(\gamma) < 0$ implies $\dot z_1(\gamma)=\dot z_2 (\gamma)=h(\gamma)-c$ from \eqref{e:dot z i}. Hence, $z$ is a solution of \eqref{e:problem0}.

\smallskip

{\em Assume $\alpha=\gamma$}. In this case, we have two functions $\zeta_1:(0,\alpha)\to (\xi_\alpha^1, \xi_0)$ and $\zeta_2:(\alpha,1)\to (\xi_1, \xi_\alpha^2)$ such that $\zeta_i=\phi^{-1}$ as above.

%

Consequently, define $z_1(\phi):=D(\phi)\phi\rq{}\left(\zeta_1(\phi)\right)$, $\phi \in (0,\alpha)$, and $ z_2(\phi):= D(\phi) \phi\rq{}\left(\zeta_2(\phi)\right)$, $\phi \in (\alpha, 1)$, which realize \eqref{e:problem0} in $(0,\alpha)$ and $(\alpha,1)$, respectively. In this case, again by the continuity of $\mathcal V$, we have directly that the function $z$ obtained by pasting together $z_1$ and $z_2$ is a solution of \eqref{e:problem0}. Hence, we proved $(i)$.

\smallskip

{\em Assume $\alpha < \gamma$}. In this case, we have
$-\infty \le \xi_1 < \xi_\gamma^2 \le \xi_\gamma^1< \xi_\alpha  < \xi_0\le \infty$.
%
%

Thus, we define three functions $\zeta_1:(0,\alpha)\to (\xi_\alpha,\xi_0)$, $\zeta_2:(\alpha,\gamma)\to (\xi_\gamma^1, \xi_\alpha)$ and $\zeta_3:(\gamma,1)\to (\xi_1, \xi_\gamma^2)$ such that $\zeta_i=\phi^{-1}$, as above. Let $z_i$, for $i=1,2,3$, be defined by
\begin{align*}
\label{zetaaa}
&z_1(\phi):=D(\phi)\phi\rq{}\left(\zeta_1(\phi)\right), \ \phi \in (0,\alpha),\\
& z_2(\phi):= D(\phi) \phi\rq{}\left(\zeta_2(\phi)\right), \ \phi \in (\alpha,\gamma),\\
& z_3(\phi):=D(\phi)\phi\rq{}\left(\zeta_3(\phi)\right), \ \phi \in (\gamma,1).
\end{align*}

Analogously to the case $\alpha>\gamma$, we can paste together in $[0,1]$ the continuous extensions of $z_1,z_2,z_3$ to obtain a function $z \in C^0[0,1]$ which satisfies $\eqref{e:problem0}_1$ in $(0,\alpha)\cup (\alpha,\gamma)\cup (\gamma,1)$, $\eqref{e:problem0}_2$ and $\eqref{e:problem0}_4$. Hence, $z$ satisfies \eqref{e:problem0prime}.
\par
Assume now \eqref{g sublinear gamma}. To show that $z$ satisfies \eqref{e:problem0}, it remains to prove that $z(\gamma)>0$ and that $z$ is $C^1$ in $\gamma$. To this end, assume by contradiction $z(\gamma)=0$, which is $z_2(\gamma)=z_3(\gamma)=0$. Since $D(\gamma)\neq 0$, this means that $\phi\rq{}({\xi_\gamma^1}^+)=\phi\rq{}({\xi_\gamma^2}^-) =0$. Focus on $\phi \in (\alpha, \gamma)$. The function $\psi=\psi(\xi)$ defined by
\[ \psi(\xi):=\begin{cases} \phi(\xi) \ \mbox{ for } \xi \in (\xi_\gamma^1, \xi_\alpha),\\ \gamma \ \mbox{ for } \ \xi \in (-\infty,\xi_\gamma^1],\end{cases}\]
is the profile of a strict semi-wavefront from $\gamma$, connecting $\gamma$ to $\alpha$, for the equation
\begin{equation} \label{e:E2}\rho_t - f(\rho)_x = \left(-D(\rho)\rho_x\right)_x-g(\rho)\end{equation}
 restricted to the interval $\rho \in [\alpha,\gamma]$ and with speed $-c$. Since $-D>0$ in $(0,\gamma]$ and $-g>0$ in $[0,\gamma)$ with $-g(\gamma)=0$, Equation \eqref{e:E2} fits with \cite[Theorem 2.9]{CM-DPDE}. Condition \eqref{g sublinear gamma} (in fact only its left-side part suffices) and \cite[Theorem 2.9]{CM-DPDE} imply that $\psi$ must be necessarily strictly monotone at $\gamma$, which is clearly a contradiction. Assuming \eqref{g sublinear gamma} in a right-hand neighborhood of $\gamma$, similar arguments lead to a contradiction as well.

Hence, as in the case $\alpha > \gamma$, the fact that $z(\gamma)\neq 0$ and the continuity of $\mathcal{V}$ imply that there exists a unique $\xi_\gamma$ such that $\phi(\xi_\gamma)=\gamma$. This fact, as in the case $\alpha>\gamma$, implies that $\dot z(\gamma^\pm)=h(\gamma)-c$ and in turn that $z \in C^1(0,\alpha)\cup (\alpha,1)$. This concludes the proof of Part $(ii)$.
\par
The last part of the statement regarding $\xi_\gamma$ follows from the regularity of $\phi$ and the fact that $z(\gamma)/D(\gamma)<0$.
\end{proof}

The following result is essentially the converse of Proposition \ref{prop:nec}.

\begin{proposition}
\label{prop:suff}
For $c\in\R$, assume one of the following:
\begin{enumerate}[(a)]
\item there exists a solution of \eqref{e:problem0} and either
\begin{enumerate}[(i)]
\item  $\alpha \neq \gamma$, or

\item  $\alpha =\gamma$, and \eqref{integrability g} holds;
 \end{enumerate}
 \item there exists a solution of \eqref{e:problem0prime} with $z(\gamma)=0$, $\alpha<\gamma$, and \eqref{integrability g} holds.
 \end{enumerate}
 Then, Equation  \eqref{e:E} admits a wavefront associated to $c$ whose profile $\phi$ satisfies \eqref{e:infty}.
\end{proposition}

We postpone the proof of Proposition \ref{prop:suff} to Section \ref{s:proofs}. The following corollary establishes a correspondence between profiles $\phi$ and solutions $z$ of \eqref{e:problem0} or \eqref{e:problem0prime}.

\begin{corollary}
\label{cor:equivalence}
We have:
\begin{enumerate}[(a)]
\item when $\alpha >\gamma$, there exists a one-to-one correspondence between profiles $\phi$ (up to space shifts) with speed $c$ and solutions $z$ of \eqref{e:problem0};
\item when $\alpha=\gamma$, under \eqref{integrability g} there exists a one-to-one correspondence between profiles $\phi$ (up to space shifts and stretchings at $\gamma$) with speed $c$ and solutions $z$ of \eqref{e:problem0};
\item when $\alpha < \gamma$ and \eqref{g sublinear gamma} holds, then there exists a one-to-one correspondence between profiles $\phi$ (up to space shifts) with speed $c$ and solutions $z$ of \eqref{e:problem0}. Instead, under \eqref{integrability g} there exists a one-to-one correspondence between profiles $\phi_0$ of Theorem \ref{th:multiplicity} $(c)$ (up to shifts and stretchings at $\gamma$) and solutions of \eqref{e:problem0prime} with $z(\gamma)=0$.
\end{enumerate}
\end{corollary}

\begin{proof}
We deduce the statement by a direct application of Propositions \ref{prop:nec} and \ref{prop:suff}.
\end{proof}


\section{Solvability of the first-order problem}
\label{s:first order}
\setcounter{equation}{0}

Motivated by Propositions \ref{prop:nec}, \ref{prop:suff} and Corollary \ref{cor:equivalence}, in this section we take up the study of problem \eqref{e:problem0}. We first collect the main properties of the solutions to the general problem
\begin{equation}
\label{e:problem0101}
\begin{cases}
\dot z(\phi) = h(\phi)-c- \frac{D(\phi)g(\phi)}{z(\phi)}, \ &\phi \in (\sigma_1,\sigma_2),\\
z(\phi)<0, \ &\phi \in (\sigma_1, \sigma_2).
\end{cases}
\end{equation}
We denote in the following $q:=Dg$. We refer to Figure 2.

We recall that for a function $q:[0,1]\to \R$, the notation $D_{+}q(\rho_0)$ and $D_{-}q(\rho_0)$, with $\rho_0\in[0,1]$, stands for the {\em right}, resp., {\em left lower Dini-derivative} of $q$ at $\rho_0$; analogously, $D^{\pm}q$ represent the {\em right} and {\em left upper Dini-derivatives} of $q$.
More explicitly,
\[
D_{\pm}q(\rho_0):=\liminf_{\rho \to \rho_0^\pm}\frac{q(\rho)-q(\rho_0)}{\rho-\rho_0},
\quad
D^{\pm}q(\rho_0):=\limsup_{\rho \to \rho_0^\pm}\frac{q(\rho)-q(\rho_0)}{\rho-\rho_0}.
\]

\begin{lemma}
\label{lem:0101}
Let $h,q$ be continuous functions on $[\sigma_1, \sigma_2]$, with $q >0$ in $(\sigma_1, \sigma_2)$ and $q(\sigma_1)=q(\sigma_2)=0$. Consider Problem \eqref{e:problem0101}. Then we have:
\begin{enumerate}[(a)]

\item For every $c \in \R$, there exists a (unique) $\zeta_c \in C^0[\sigma_1,\sigma_2] \cap C^1\left(\sigma_1,\sigma_2\right)$ satisfying \eqref{e:problem0101} and such that $\zeta_c(\sigma_2)=0$.

\item If $c_1 < c_2$ then $\zeta_{c_1}(\phi) < \zeta_{c_2}(\phi)$, for $\phi \in (\sigma_1,\sigma_2)$. Moreover, if $\zeta_{c_1}(\sigma_1)<0$ then $\zeta_{c_1} < \zeta_{c_2}$ in $[\sigma_1, \sigma_2)$.

\item It holds that $\lim_{c\to \infty}\zeta_c(\sigma_1)=0$.

\item Let $c^*=c^*\left(q; h; \sigma_1,\sigma_2\right)\in (-\infty,\infty]$ be defined by
\begin{equation}
\label{c*}
c^*:=\sup \left\{c\in \R: \zeta_c(\sigma_1)<0\right\}.
\end{equation}
If $c^* < \infty$, then for every $c>c^*$, there exists $\beta(c) \in (-\infty,0)$ such that there is a (unique) $z_{c,s} \in C^0[\sigma_1,\sigma_2]\cap C^1\left(\sigma_1,\sigma_2\right]$ satisfying \eqref{e:problem0101} and both $z_{c,s}(\sigma_1)=0$ and $z_{c,s}(\sigma_2)=s<0$, if and only if $s \ge \beta(c)$.

\end{enumerate}
\end{lemma}

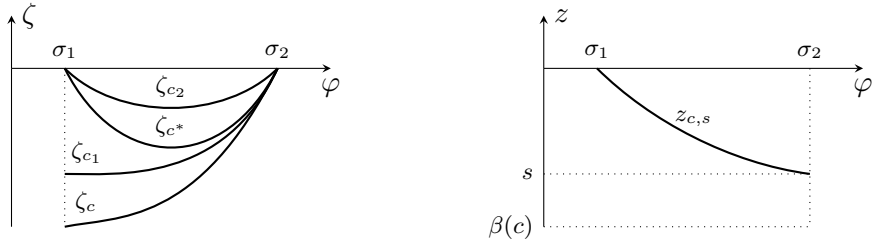
\begin{figure}[htb]
\begin{center}
\begin{tikzpicture}[>=stealth, scale=0.7]
\draw[->] (0,0) --  (6,0) node[below]{$\phi$} coordinate (x axis);
\draw[->] (0,0) -- (0,1) node[right]{$\zeta$} coordinate (y axis);
\draw (0,0) -- (0,-3);
\draw[thick] (1,-3) .. controls (2,-2.8) and (3.5,-3) .. (5,0) node[above]{\footnotesize{$\sigma_2$}} node[very near start, above]{\footnotesize$\zeta_c$}; 
\draw[dotted] (1,0) node[above]{\footnotesize$\sigma_1$} --(1,-3);
\draw[thick] (1,-2) .. controls (2,-2) and (4,-2.2) .. (5,0) node[very near start, above]{\footnotesize$\zeta_{c_1}$}; 
\draw[thick] (1,0) .. controls (2,-2) and (4,-2) .. (5,0) node[midway, above]{\footnotesize$\zeta_{c^*}$}; 
\draw[thick] (1,0) .. controls (2,-1) and (4,-1) .. (5,0) node[midway, above]{\footnotesize$\zeta_{c_2}$}; 

\begin{scope}[xshift=10cm]
\draw[->] (0,0) --  (6,0) node[below]{$\phi$} coordinate (x axis);
\draw[->] (0,0) -- (0,1) node[right]{$z$} coordinate (y axis);
\draw (0,0) -- (0,-3);
\draw[thick] (1,0) node[above]{\footnotesize{$\sigma_1$}} .. controls (2,-1) and (3.5,-1.8) .. (5,-2) node[midway, above]{\footnotesize$z_{c,s}$}; 
\draw[dotted] (5,0) node[above]{\footnotesize$\sigma_2$} --(5,-3);
\draw[dotted] (0,-2) node[left]{\footnotesize$s$}-- (5,-2);
\draw[dotted] (0,-3) node[left]{\footnotesize$\beta(c)$}-- (5,-3);

\end{scope}

\end{tikzpicture}
\end{center}
\caption{\label{f:11}{Plots of the solutions $\zeta_c$ (left, for $c<c_1<c^*<c_{3.2}$) and $z_{c,s}$ (right, for $c>c^*$) in Lemma \ref{lem:0101}.}}
\end{figure}

\begin{proof}
It is easy to verify that $\eqref{e:problem0101}$ is equivalent to the same problem in $[0,1]$. Part $(a)$ then follows  from \cite[Theorem 2.6]{CM-DPDE} while the first part of
$(b)$, i.e., that $c_1<c_{2}$ implies $\zeta_{c_1} < \zeta_{c_2}$ in $(\sigma_1, \sigma_2)$ and hence $\zeta_{c_1} (\sigma_1)\le\zeta_{c_2}(\sigma_1)$, was proved in \cite[Lemma 5.1]{CM-DPDE}. Assume $\zeta_{c_1}(\sigma_1)=\zeta_{c_2}(\sigma_1)<0$, hence $\dot {\zeta_{c_1}}(\sigma_1^+)=h(\sigma_1)-c_1 > h(\sigma_1)-c_{2} =\dot {\zeta_{c_2}}(\sigma_1^+)$ and we get a contradiction with $\zeta_{c_1} < \zeta_{c_2}$ in $(\sigma_1, \sigma_2)$; part $(a)$ is proved.
\par
We prove $(c)$. Let $\{c_n\}_n$ be an increasing sequence with $c_n \to \infty$, as $n \to \infty$. Define
\[ \mu := \lim_{n \to \infty} \zeta_{c_n}(\sigma_1).\]
Clearly, $\mu \le 0$. Suppose by contradiction that $\mu <0$ and denote with $\eta=\eta(\phi)$  the solution of the Cauchy problem
\[ \begin{cases}
\dot \eta(\phi) = h(\phi)-\delta -\frac{q(\phi)}{\eta(\phi)}, \ \phi >\sigma_1,\\
\eta(\sigma_1)=\mu,
\end{cases}
\]
defined in $[\sigma_1, \sigma)$ with  $\sigma\le \sigma_2$ and $\delta<c_1$. Since $\eta$ satisfies
\[
\dot{\eta} > h-c_n -\frac{q}{\eta} \ \mbox{ in } \ (\sigma_1,\sigma) \ \mbox{ and } \ \eta(\sigma_1) \ge \zeta_{c_n}(\sigma_1),
\]
 by a comparison-type argument (see e.g. \cite[Lemma 3.2 (2.a.ii)]{BCM1}), it follows $\eta > \zeta_{c_n}$ in $(\sigma_1,\sigma)$. The latter implies, for every $[a,b] \subset (\sigma_1, \sigma)$ and for every $n \in \N$,
\[  \zeta_{c_1}(b)-\eta(a) < \zeta_{c_n}(b)-\zeta_{c_n}(a) < \int_{a}^{b} h(\tau) -c_n -\frac{q(\tau)}{\eta(\tau)}\, d\tau\le (b-a)\left(\max_{[a,b]}\left|h -\frac{q}{\eta}\right| -c_n\right),\]
which tends to $-\infty$ as $n\to\infty$. This is a contradiction, since $\zeta_{c_1}$ and $\eta$ are continuous in $[a,b]$.
\par
Lastly, part {\em (d)} is deduced from \cite[Proposition 5.1]{BCM1}. In fact, in \cite[Proposition 5.1]{BCM1} it was also assumed (in our notation) $D^+ q(\sigma_1)<\infty$, in order to have $c^* <\infty$. Here, this is not necessary since we already assumed $c^*<\infty$.
\end{proof}

The following corollary establishes some fundamental estimates for the threshold $c^*$. We refer to \cite{Crooks} for some new estimates on $c^*$ when $D=1$ and $f=0$.

\begin{corollary}
\label{cor:estimates}
We make the same assumptions of Lemma \ref{lem:0101}. Then we have
 \begin{equation}
 \label{e:c*below}
  c^*\left(q;h;\sigma_1,\sigma_2\right) \ge \max \left\{\sup_{(\sigma_1, \sigma_2]}\delta\left(f, \sigma_1\right), h(\sigma_1)+2\sqrt{D_+q(\sigma_1)}\right\}\ge h(\sigma_1),
 \end{equation}
and, if $h(\phi)\ge h(\sigma_1)$ in a right neighborhood of $\sigma_1$ then
 \begin{equation}
 \label{e:c*below2}
 c^*\left(q;h;\sigma_1,\sigma_2\right) > h(\sigma_1).
 \end{equation}
Moreover
 \begin{equation}
 \label{e:c*above}
 c^*\left(q;h;\sigma_1,\sigma_2\right) \le \sup_{(\sigma_1, \sigma_2]}\delta\left(f, \sigma_1\right) + 2 \sqrt{\sup_{(\sigma_1, \sigma_2]}\delta\left(q, \sigma_1\right)},
 \end{equation}
and $c^* \in \R$ if $D^+q(\sigma_1)<\infty$.
\end{corollary}

\begin{proof}
Formula \eqref{e:c*below} follows from \cite[formula (5.5)]{BCM1}, Formula  \eqref{e:c*below2} from \cite[Remark 6.4]{BCM1} (see also references therein) and  \eqref{e:c*above} from \cite[Lemma 4.1]{BCM1}.
\end{proof}

\begin{remark}
{
\rm
By \cite[Theorem 3.1]{Marcelli-Papalini}, we know that if $q$ is differentiable at $\sigma_1$ then $c\in \R$ and
  \begin{equation}
 \label{e:c*above2}
  c^*\left(q;h;\sigma_1,\sigma_2\right) \le \sup_{(\sigma_1, \sigma_2]}\delta\left(f, \sigma_1\right) + 2 \sqrt{\sup_{\phi \in (\sigma_1, \sigma_2]}\frac1{\phi -\sigma_1}\int_{\sigma_1}^{\phi}\frac{q(\sigma)}{\sigma-\sigma_1}\,d\sigma}.
 \end{equation}
}
\end{remark}


Now, on the basis of the general results given in Lemma \ref{lem:0101} and Corollary \ref{cor:estimates}, we are going to consider several sub-problems of \eqref{e:problem0}  and perform several change of variables; for brevity, for functions $F$ and $\omega\in\R$ we use the notation
\begin{equation}
\label{changes}
\bar{F}(\phi)=F(1-\phi), \quad \tilde{F}(\phi) = -F(1-\phi), \quad \hbox{ and }\quad \bar \omega = 1-\omega.
\end{equation}

The results depend on the mutual positions of the points $\alpha$ and $\gamma$ and then we consider separately the cases $\alpha>\gamma$, $\alpha=\gamma$ and $\alpha<\gamma$.
The discussion involves the values
\begin{equation}
\label{e:soglie1}
\begin{array}{ll}
c_{1.1}^*:=-c^*\left(\tilde{q}; \tilde{h}; \max\{\bar{\alpha},\bar{\gamma}\},1\right);  &c_{1.2}^*:=c^*\left(q; h; \gamma,\alpha\right), \ \mbox{ if } \ \alpha>\gamma;
\\
c_{3.1}^*:=-c^*\left(q; -h; \alpha,\gamma\right), \ \mbox{ if } \ \alpha<\gamma; \quad&
c_{3.2}^*:=c^*\left(\tilde q; \bar h; 0,\min\{\bar\alpha,\bar\gamma\}\right).
\end{array}
\end{equation}
Here $c^*$ is defined as in \eqref{c*} and $\tilde{q}, \tilde{h}, \bar{h}, \bar{\gamma}$ and $\bar{\alpha}$ as in \eqref{changes}.

\subsection {The subcase $\alpha>\gamma$}

We start with the sub-problem in $(0, \alpha)$, that is we consider
\begin{equation}
\label{e:problem1}
\begin{cases}
\dot{z}=h-c- q/z \ &\mbox{ in } \ (0,\alpha),\\
z<0 \ &\mbox{ in } \ (0,\alpha),\\
z(0)=z(\alpha)=0.
\end{cases}
\end{equation}

\begin{lemma}
\label{lem:problem1}
Problem \eqref{e:problem1} admits a solution $z\in C^0[0,\alpha]\cap C^1(0,\alpha) $ if and only if $c_{1.1}^* < c_{1.2}^*$. In such a case, solutions only occur for a unique real value $c=c_1^*\in (c_{1.1}^* , c_{1.2}^*)$, and the corresponding solution $z$ is unique.
\end{lemma}

\begin{proof}
Consider the problems
\begin{equation}
\label{e:problem2}
\begin{cases}
\dot{z_1}=h-c- q/z_1 \ &\mbox{ in } \ (0,\gamma),\\
z_1<0 \ &\mbox{ in } \ (0,\gamma],\\
z_1(0)=0,
\end{cases}
\qquad
\begin{cases}
\dot{z_2}=h-c- q/z_2 \ &\mbox{ in } \ (\gamma,\alpha),\\
z_2<0 \ &\mbox{ in } \ [\gamma,\alpha),\\
z_2(\alpha)=0
\end{cases}
\end{equation}
and assume that there exists $z$ which satisfies \eqref{e:problem1} for some $c\in \R$. Thus, the restrictions of $z$ to $(0,\gamma)$ and $(\gamma,\alpha)$ satisfy  (respectively) $\eqref{e:problem2}_1$ and $\eqref{e:problem2}_2$.

Note that $z$ satisfies $\eqref{e:problem2}_1$ if and only if $\bar z=\bar z(\phi)$, defined for $\phi \in [\bar\gamma,1]$ as in \eqref{changes}, satisfies $\bar z(1)=0$ and $\eqref{e:problem0101}$ in $(\bar \gamma,1)$ with $\tilde h$, $\tilde q$ and $C=-c$, in place of $h$, $q$ and $c$. Since $\bar z(\bar\gamma)$ must equal $z(\gamma)<0$, we apply then Lemma \ref{lem:0101} $(d)$ to infer that $\bar z$ (and hence $z$) exists if and only if $-c < c^*(\tilde q;\tilde h; \bar\gamma,1)$. By \eqref{e:soglie1}, we obtain that $c>c_{1.1}^*$.
\par
We get $c<c_{1.2}^*$ directly from Lemma \ref{lem:0101} $(d)$ in $(\gamma,\alpha)$,  since $z$ satisfies $\eqref{e:problem2}_2$.
\par
Vice versa, assume $c_{1.1}^*<c_{1.2}^*$.  We define the function $F=F(c)$ by
\[ F(c):= z_1(c;\gamma) - z_2(c;\gamma), \ c \in (c_{1.1}^*, c_{1.2}^*),\]
where $z_1(c;\cdot)$ and $z_2(c; \cdot)$ are the solutions of $\eqref{e:problem2}_1$ and $\eqref{e:problem2}_2$; $F$ is well-defined by Lemma \ref{lem:0101} $(a)$ and $(d)$.  The aim of $F$ is now explained. Problem \eqref{e:problem1} has a solution associated to some $c\in\R$ if and only if $F(c)=0$. We claim that $F$ is strictly decreasing, $F$ is continuous and satisfies
\begin{equation}
\label{F changes sign}
 \lim_{c \to \left\{c_{1.1}^*\right\}^+}F(c) >0 \ \mbox{ and } \ \lim_{c \to \left\{c_{1.2}^*\right\}^-} F(c) < 0.
 \end{equation}
 Lemma \ref{lem:0101} $(b)$ applied to $\eqref{e:problem2}_2$ implies that $c\mapsto z_2(c;\gamma)$ is strictly increasing, since $c < c_{1.2}^*$. Instead, the application of Lemma \ref{lem:0101} $(b)$ to $\bar z_1$ implies that $c\mapsto z_1(c;\gamma)$ is strictly decreasing. Hence, $F$ is strictly decreasing in its domain, see Figure \ref{f:fz1z2}.

\begin{figure}[htb]
\begin{center}
\begin{tikzpicture}[>=stealth, scale=0.7]
\draw[->] (0,0) --  (6,0) node[below]{$\phi$} coordinate (x axis);
\draw[->] (0,0) -- (0,1) node[right]{$z$} coordinate (y axis);
\draw (0,0) -- (0,-2);
\draw[thick] (0,0) .. controls (0.5,-1.3) and (1.5,-1.4) .. (2,-1.5) node[midway, below]{$z_1$}; 
\draw[->] (1.8,-1.6) -- (1.8,-2.3);
\draw[thick] (2,-2) .. controls (2.5,-1.9) and (3.5,-1.5) .. (4,0) node[above]{\footnotesize{$\alpha$}} node[midway, below]{$z_2$};
\draw[->] (2.2,-1.8) -- (2.2,-1.1);
\draw[dotted] (2,0) node[above]{$\gamma$}-- (2,-2);
\draw (5,-0.1)--(5,0.1);
\draw (5,0) node[above]{\footnotesize{$1$}};
\end{tikzpicture}
\end{center}
\caption{\label{f:fz1z2}{Typical plots of the solutions $z_1$ to $\eqref{e:problem2}_1$ and $z_2$ to $\eqref{e:problem2}_2$. The arrows denote how curves changes as $c$ increases, see Lemma \ref{lem:problem1}.}}
\end{figure}
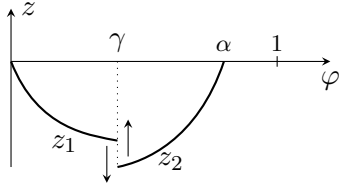

We prove now that $F$ is continuous. Fix $c_{1.1}^* < c < c_{1.2}^*$ and let $\{c_n\}_n$ be an increasing sequence converging to $c$. Define $z_n(\phi):=z_2(c_n;\phi)$, for $\phi \in [\gamma, \alpha]$. From Lemma \ref{lem:0101} $(b)$, $\{z_n\}$ is increasing in $[\gamma, \alpha)$ and is bounded from above by $z_2(c;\cdot)<0$. By applying \cite[Lemma 3.3]{BCM1} we deduce that $z_n$ converges in $[\gamma,\alpha]$ to $z_2(c;\cdot)$. In particular, $z_n(\gamma) \to z_2(c;\gamma)$. The same conclusion holds if $\{c_n\}_n$ is decreasing to $c$ (again for \cite[Lemma 3.3]{BCM1}). By repeating the same argument for $\bar z_1$, we then proved that $F$ is continuous. Lastly, Formula \eqref{F changes sign} follows from Lemma \ref{lem:0101} $(c)$. This completely proves our claim.

Hence, there exists a unique real $c=c_1^*$ such that $F(c)=0$ and the proof of the lemma is completed.
\end{proof}

\begin{remark}
{
\rm
Estimates for $c_1^*$ follow from those for $c_{1.1}^*$ and $c_{1.2}^*$, which are in turn obtained from \eqref{e:soglie1} and Corollary \ref{cor:estimates}.
}
\end{remark}

We give the following result for \eqref{e:problem0}.

\begin{proposition}
\label{prop:1}
When $\alpha>\gamma$, Problem \eqref{e:problem0} admits a solution if and only if $c_{1.1}^*<c_{1.2}^*$ and $c_1^* \ge c_{3.2}^*$, where $c_1^*$ is given in Lemma \ref{lem:problem1}. In such a case, we necessarily have $c=c_1^*$ and the solution is unique.
\end{proposition}

\begin{proof}
Assume that \eqref{e:problem0} is solvable, for some value $c$. Then also \eqref{e:problem1} is solvable, for the same $c$, implying $c_{1.1}^*<c_{1.2}^*$ and $c=c_1^*$, by Lemma \ref{lem:problem1}. Notice that  $w(\phi):=\tilde{z}(\phi)$, defined as in \eqref{changes} , satisfies \eqref{e:problem0101} in $(0,\bar\alpha)$, with $\bar h, \tilde q$; hence, by Lemma \ref{lem:0101} $(d)$, we obtain that $c_1^* \ge c_{3.2}^*$. The converse implication can be proved in a similar way. The solution is unique by Lemma \eqref{e:problem0101}.
\end{proof}


\subsection {The subcase $\alpha=\gamma$}

%


\begin{proposition}
\label{prop:2}
When $\alpha=\gamma$, then $c_{1.1}^*,c_{3.2}^*\in\R$ and Problem \eqref{e:problem0} admits solutions if and only if $c_{3.2}^* \le c_{1.1}^*$. In such a case, a (unique) solution exists for every $c_{3.2}^* \le c \le c_{1.1}^*$.
\end{proposition}

\begin{proof}
First, note that since $\dot {\tilde q} (\bar\alpha)=\dot D(\alpha) g(\alpha) =0$ then from Corollary \ref{cor:estimates} we deduce that $c_{1.1}^*\in \R$. Same arguments  imply that $c_{3.2}^* \in \R$.
Let $z$ be a solution to \eqref{e:problem0}, for some value $c$. Then $z$ is also a solution to \eqref{e:problem1}, for the same $c$. Hence the function $\bar{z}$ defined for $\phi \in [\bar\alpha,1]$ as in \eqref{changes}, satisfies $\bar z(1)=0$ and $\eqref{e:problem0101}$ in $(\bar \alpha,1)$ with $\tilde h$, $\tilde q$ and $-c$, in place of $h$, $q$ and $c$. We apply then Lemma \ref{lem:0101} $(d)$ to infer that $\bar z$ (and hence $z$) exists if and only if $-c \ge c^*(\tilde q;\tilde h; \bar\gamma,1)$. By \eqref{e:soglie1}, we obtain that $c\le c_{1.1}^*$.
\par Consider now the interval $[\alpha, 1]$; by applying  to $z$ the same arguments in Proposition \ref{prop:1} we get $c\ge c_{3.2}^*$. Hence $c_{3.2}^* \le c_{1.1}^*$ and $c_{3.2}^* \le c \le c_{1.1}^*$.  At last $z$ is unique, by Lemma \ref{lem:0101} $(a)$.

\par The converse implication can be obtained similarly.
\end{proof}

\subsection {The subcase $\alpha<\gamma$}
%
%
First, we restrict our attention to the sub-problem
\begin{equation}
\label{e:problem3}
\begin{cases}
\dot{z}=h-c-q/z \ &\mbox{ in } \ (\alpha,1),\\
z>0 \ &\mbox{ in } \ (\alpha,1),\\
z(\alpha)=z(1)=0,
\end{cases}
\end{equation}
which we further decompose as in the following lemma.

\begin{lemma}
\label{lem:c6}
We consider:
\begin{equation}
\label{e:problem6}
\begin{cases}
\dot{z_1}(\phi)=h(\phi)-c-\frac{q(\phi)}{z_1(\phi)}, \ &\phi\in (\alpha,\gamma),\\
z_1(\phi)>0 \ &\phi \in (\alpha,\gamma),\\
z_1(\alpha)=0,
\end{cases}
\qquad
\begin{cases}
\dot{z_2}(\phi)=h(\phi)-c-\frac{q(\phi)}{z_2(\phi)}, \ &\phi\in (\gamma,1),\\
z_2(\phi)>0,  \ &\phi \in (\gamma,1),\\
z_2(1)=0.
\end{cases}
\end{equation}
We have:
\begin{enumerate}[(1)]

\item Problem $\eqref{e:problem6}_1$ has solutions if and only if $c\le c_{3.1}^*$. For a given $c < c_{3.1}^*$ there exist infinitely many solutions $z_1$, one for each value $z_1(\gamma)$ in the interval $(0,\beta_1(c)]$, for some $\beta_1(c)>0$.

\item Problem $\eqref{e:problem6}_2$ has solutions if and only if $c\ge c_{3.2}^*$. For a given $c > c_{3.2}^*$ there exist infinitely many solutions $z_2$, one for each value $z_2(\gamma)$ in the interval $(0,\beta_2(c)]$, for some $\beta_2(c)>0$.
\end{enumerate}
\end{lemma}

\begin{proof}
We first prove (1). We apply Lemma \ref{lem:0101} $(d)$ to $-z_1$. Indeed, $z_1$ satisfies $\eqref{e:problem6}_1$ if and only if $-z_1(\alpha)=0$ and $-z_1$ satisfies \eqref{e:problem0101} in $(\alpha,\gamma)$ with $-h$ and $-c$ in place of $h$ and $c$.  Hence, from Lemma \ref{lem:0101} $(d)$, such a $-z_1$ exists if and only if $-c \ge c^*\left(q; -h; \alpha,\gamma\right)$. Moreover, if $-c >c^*\left(q; -h; \alpha,\gamma\right)$ then $-z_1(\gamma)=s<0$ if and only if $s\ge \beta(-c)$, where $\beta(-c)<0$ is given in Lemma \ref{lem:0101} $(d)$
 and $\beta_1(c):=-\beta(-c)$ and by condition \eqref{c*} we complete Part $(1)$.

About (2), a function $z_2$ satisfies $\eqref{e:problem6}_2$ if and only if $\tilde {z_2}$ (defined as in \eqref{changes}) satisfies \eqref{e:problem0101} in $(0,\bar\gamma)$ with $\tilde q$ and $\bar h$ in place of $q$ and $h$ and $\tilde {z_2}(0)=0$. Therefore, Part $(2)$ follows by Lemma \ref{lem:0101} $(d)$ with $\beta_2(c):=-\beta(c)$, where $\beta(c)$ is relative to \eqref{e:problem0101} in $(0,\bar\gamma)$ with $\tilde q$ and $\bar h$.
\end{proof}

\begin{lemma}
\label{lem:c6final}
Problem \eqref{e:problem3}  has a solution if $c_{3.2}^* < c_{3.1}^*$ and $c_{3.2}^*< c < c_{3.1}^*$. Moreover, for every $c$ varying in such a range, define
\begin{equation}
\label{e:beta}
\beta(c):=\min\left\{\beta_1(c),\beta_2(c)\right\}.
\end{equation}
Then, there is a one-parameter family of solutions given by
\begin{equation}\label{e:zc2}
\mathcal{Z}_c=\left\{z_s: z_s(\gamma)=s, \ s\in (0,\beta(c)]\right\}.
\end{equation}
Conversely, if Problem \eqref{e:problem3} has a solution then $c_{3.2}^*\le c_{3.1}^*$ and $c_{3.2}^*\le c \le c_{3.1}^*$.
\end{lemma}

\begin{proof}
Assume  $c_{3.1}^*> c_{3.2}^*$ and pick $c\in (c_{3.2}^*,c_{3.1}^*)$.
Fix also $0<s\le \beta(c)$. From Lemma \ref{lem:c6}, there are both a solution $z_1$ of $\eqref{e:problem6}_1$ and a solution $z_2$ of $\eqref{e:problem6}_2$ such that $z_1(\gamma)=s=z_2(\gamma)$. In addition, because both $q/z_1$ and $q/z_2$ tend  to $0$ when their arguments tend to $\gamma^\mp$, from the differential equations in $\eqref{e:problem6}_1$ and $\eqref{e:problem6}_2$ we obtain
$$
\lim_{\phi \to \gamma^-}
\dot{z_1}(\phi)=\lim_{\phi \to \gamma^+}
\dot{z_2}(\phi)=h(\gamma)-c.
$$
Define $z=z_1$ in $[\alpha,\gamma]$ and $z=z_2$ in $(\gamma,1]$. The function $z$ is a solution of \eqref{e:problem3} which satisfies $z(\gamma)=s$ and then the existence of the family $\mathcal{Z}_c$ in \eqref{e:zc2} follows as well.

It remains to prove the last part of the statement. Suppose that a solution $z$ of \eqref{e:problem3} exists for some $c$. Necessarily, we must have $z(\gamma)>0$ by $\eqref{e:problem3}_2$. Thus, setting $z_1(\phi)=z(\phi)$, for $\phi \in [\alpha,\gamma]$ and $z_2(\phi)=z(\phi)$, for $\phi \in [\gamma,1]$, we deduce that $z_1$ and $z_2$ satisfy $\eqref{e:problem6}_1$ and $\eqref{e:problem6}_2$, respectively. Hence, we deduce $c\le c_{3.1}^*$ from Lemma \ref{lem:c6} (1) and  $c\ge c_{3.2}^*$ from Lemma \ref{lem:c6} (2). Hence, $c_{3.1}^*\ge c_{3.2}^*$.
\end{proof}


Lastly, we have the following result.

\begin{proposition}
\label{prop:3}
When $\alpha<\gamma$ we have $c_{1.1}^* \in \R$. Moreover, Problem \eqref{e:problem0} admits solutions if $c_{3.2}^*<\min\{c_{1.1}^*, c_{3.1}^*\}$ and no solutions if  $c_{3.2}^*>\min\{c_{1.1}^*, c_{3.1}^*\}$. In the former case, solutions exist for every $c$ such that
\begin{equation}
\label{e:range c}
c\in
\left\{
\begin{array}{ll}
(c_{3.2}^*, c_{1.1}^*] & \mbox{ if } \ c_{1.1}^* < c_{3.1}^*,
\\[2mm]
(c_{3.2}^*, c_{3.1}^*) & \mbox{ if } \ c_{1.1}^* \ge c_{3.1}^*.
\end{array}
\right.
\end{equation}
For any fixed such $c$, there is a one-parameter family of solutions in one-to-one correspondence to the family \eqref{e:zc2}.
\par
\end{proposition}

\begin{proof}
First, since $\dot{{q}}(\alpha)=\dot{D}(\alpha)g(\alpha)$, $c_{1.1}^* \in \R$ (see  Corollary \ref{cor:estimates}).
If we define formally
$$J=(c_{3.2}^*,c_{3.1}^*) \cap (-\infty,c_{1.1}^*],$$
we have that $J$ is not empty if and only if $c_{3.1}^*>c_{3.2}^*$ and $c_{1.1}^* > c_{3.2}^*$. In such a case, for every $c\in J$ there are solutions $z$ of \eqref{e:problem0}, given by pasting together the solution in $(0,\alpha)$, provided as in the first half of the proof of Proposition \ref{prop:2}, and one of the representative of the family of solutions in $(\alpha,1)$, given by Lemma \ref{lem:c6final}; see Figure \ref{f:zmanyalpha}. By Lemma \ref{lem:c6final} it follows as well that if Problem \eqref{e:problem0} has solutions then $c_{3.2}^* \le \min\{c_{1.1}^*, c_{3.1}^*\}$.

\begin{figure}[htb]
\begin{center}
\begin{tikzpicture}[>=stealth, scale=0.6]
\draw[->] (0,0) --  (7,0) node[below]{$\phi$} coordinate (x axis);
\draw[->] (0,0) -- (0,2.5) node[right]{$z$} coordinate (y axis);
\draw (0,0) -- (0,-1);
\draw[thick] (0,0) .. controls (0.7,-1.5) and (1.3,-1.5) .. (2,0) node[right=5, below=0]{\footnotesize{$\alpha$}};
\draw[thick] (2,0) .. controls (4,1) and (4.5,1) .. (6,0) node[below]{\footnotesize{$1$}}; 
\draw[thick] (2,0) .. controls (4,2) and (4.5,2) .. (6,0); 
\draw[thick] (2,0) .. controls (4,3) and (4.5,3) .. (6,0); 
\draw[dotted] (4,0) node[below]{$\gamma$} -- (4,2.2);
\end{tikzpicture}
\end{center}
\caption{\label{f:zmanyalpha}{The multiple solutions $z$ of Proposition \ref{prop:3}.}}
\end{figure}
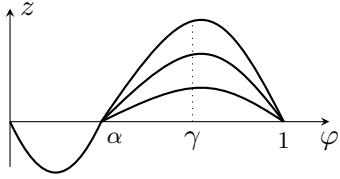

%
%
\end{proof}

We refer to Figure \ref{f:fthre} for a pictorial representation of the thresholds occurring in Propositions \ref{prop:1}, \ref{prop:2}, \ref{prop:3}.

\begin{remark}
\label{rem:(c)boundedinterval}
{
\rm
Proposition \ref{prop:3} implies that Problem \eqref{e:problem0} has solutions only if $c_{3.2}^* \le c \le \min\{c_{1.1}^*, c_{3.1}^*\}$.  Since $\dot q (\alpha)=\dot D(\alpha)g(\alpha)\in \R$, from \eqref{e:c*below} and \eqref{e:soglie1} we deduce that $c_{3.1}^*$ and $c_{1.1}^*$ must be real numbers. From  \eqref{c*} (and \eqref{e:soglie1}) applied to $c_{3.2}^*$, we deduce $ c_{3.2}^*>-\infty$. Thus, the values $c$ for which Problem \eqref{e:problem0} is solvable vary in a bounded interval.
}
\end{remark}

\begin{remark}
\label{rem:problem0prime}
{
\rm
We now discuss  the solvability of \eqref{e:problem0prime}. First, by properly adapting the proofs of Lemma \ref{lem:c6} and Proposition \ref{prop:3} we infer that Problem \eqref{e:problem0prime} has solutions if and only if $c_{3.2}^* \le \min\{c_{1.1}^*,c_{3.1}^*\}$ and if $c$ lies in that range. Secondly, for each such $c$, Problem \eqref{e:problem0prime} admits the solution satisfying $z(\gamma)=0$, which does not solve \eqref{e:problem0}.
%
}
\end{remark}

\begin{figure}[htb]
\begin{center}
\begin{tikzpicture}[>=stealth, scale=0.6]
\draw[->] (0,0) node[left]{\fs$(0,\alpha)$} --  (14,0) node[right]{$c$} coordinate (x axis); 
\filldraw[black] (6,0) circle (2pt) node[above]{$c_{1,1}^*$}; 	
\draw (5.5,0.5) node[left]{\fs$(0,\gamma)$};
\filldraw[black] (9,0) circle (2pt) node[above]{$c_1^*$};
\filldraw[black] (12,0) circle (2pt) node[above]{$c_{1,2}^*$};
\draw (12.5,0.5) node[right]{\fs$(\gamma,\alpha)$};
\draw (16,-0.4) node[right]{$\alpha>\gamma$}; 
\draw[->] (0,-0.8) node[left]{\fs$(\alpha,1)$} --  (14,-0.8) node[right]{$c$} coordinate (x axis); 
\draw[very thick] (2,-0.8) -- (14,-0.8);
\filldraw[black] (2,-0.8) circle (2pt) node[above]{$c_{3.2}^*$};
\draw[<-] (2,-1) .. controls (2.2,-2.2) and (2.4,-2.6) .. (2.8,-3.2);
\draw[<-] (3,-3.5) .. controls (3.2,-3.7) and (3.3,-5.6) .. (3.8,-6);
\draw[->] (6,-0.2) .. controls (6.1,-1.2) and (6.3,-2) .. (6.8,-2.3);
\draw[->] (7,-2.7) .. controls (7,-3.5) and (7,-4) .. (7.8,-5.1);

\begin{scope}[yshift=-2.5cm]
\draw[->] (0,0) node[left]{\fs$(0,\alpha)$} --  (14,0) node[right]{$c$} coordinate (x axis); 
\filldraw[black] (7,0) circle (2pt) node[above]{$c_{1.1}^*$};
\draw[very thick] (0,0) -- (7,0);
\draw (16,-0.4) node[right]{$\alpha=\gamma$}; 
\draw[->] (0,-0.8) node[left]{\fs$(\alpha,1)$} --  (14,-0.8) node[right]{$c$} coordinate (x axis); 
\draw[very thick] (3,-0.8) -- (14,-0.8);
\filldraw[black] (3,-0.8) circle (2pt) node[above]{$c_{3.2}^*$};
\end{scope}

\begin{scope}[yshift=-5.3cm]
\draw[->] (0,0) node[left]{\fs$(0,\alpha)$} --  (14,0) node[right]{$c$} coordinate (x axis); 
\draw[very thick] (0,0) -- (8,0);
\filldraw[black] (8,0) circle (2pt) node[above]{$c_{1.1}^*$};
\draw (16,-0.4) node[right]{$\alpha<\gamma$}; 
\draw[->] (0,-0.8) node[left]{\fs$(\alpha,1)$} --  (14,-0.8) node[right]{$c$} coordinate (x axis); 
\draw[very thick] (4,-0.8) -- (10,-0.8) node[midway, above]{$\infty$ solutions};
\draw (3.5,-0.4) node[left]{\fs$(\gamma,1)$};
\filldraw[black] (4,-0.8) circle (2pt) node[above]{$c_{3.2	}^*$};
\filldraw[black] (10,-0.8) circle (2pt) node[above]{$c_{3.1}^*$};
\draw (10.5,-0.4) node[right]{\fs$(\alpha, \gamma)$};
\end{scope}
\end{tikzpicture}
\end{center}
\caption{\label{f:fthre}{The speed thresholds in Propositions \ref{prop:1}, \ref{prop:2}, \ref{prop:3}. Pairs as $(0,\alpha)$ refer to the interval where the corresponding first-order problems are considered. Thick lines represent the values of $c$ for which the corresponding first-order problem has a solution.}}
\end{figure}
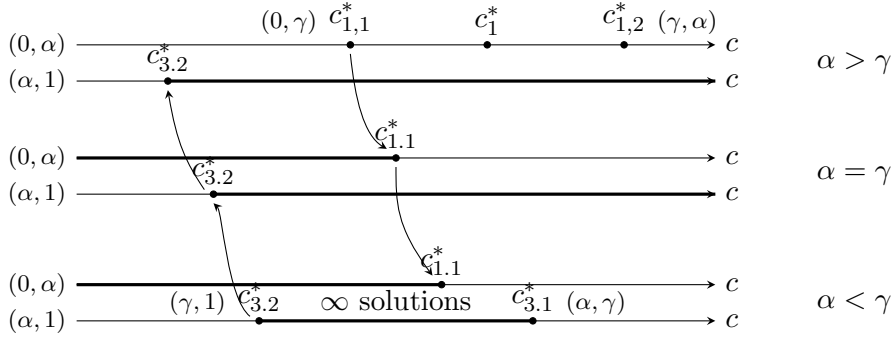

\section{Proofs of the main results}
\label{s:proofs}
\setcounter{equation}{0}

In this section we prove Proposition \ref{prop:suff} and Theorems \ref{th:multiplicity}, \ref{thm:f convex}, \ref{th:sufficient conditions}.
\begin{proof}[Proof of Proposition \ref{prop:suff}]
Assume $(a)$ and let $z$ satisfy \eqref{e:problem0}. Define $\phi_{1,\alpha}=\phi_{1,\alpha}(\xi)$ and $\phi_{\alpha,0}=\phi_{\alpha,0}(\xi)$ (see Figure \ref{f:2phis}) as the solutions of
\begin{equation}
\label{e:cauchyEQ}
\phi\rq{}=\frac{z(\phi)}{D(\phi)},
\end{equation}
with initial data (respectively)
\begin{equation}
\label{e:cauchyCI}
\phi_{1,\alpha}(0)=\frac{\alpha+1}{2} \ \mbox{ and } \ \phi_{\alpha,0}(0)=\frac{\alpha}{2}.
\end{equation}

\begin{figure}[htb]
\begin{center}

\begin{tikzpicture}[>=stealth, scale=0.3]
\draw[->] (0,0) --  (10,0) node[below]{$\xi$} coordinate (x axis);
\draw (-10,0) -- (0,0);
\draw[->] (0,0) -- (0,8) node[left]{$\phi$} coordinate (y axis);
\draw (-10,7) -- (10,7);
\draw (-10,4) -- (10,4);
\draw (-0.2,4.3) node[right]{\footnotesize{$\alpha$}};
\draw[thick] (-10,7) -- (-7,7);
\draw[dotted] (-7,0) node[below]{\footnotesize{$\xi_1$}}-- (-7,7);
\draw[thick] (-7,7) .. controls (-5,6) and (-2,6) .. (0,5.5) node[midway,below]{\footnotesize{$\phi_{1,\alpha}$}};
\draw (-0.2,5.8)  node[right]{\footnotesize{$(\alpha+1)/2$}};
\draw[thick] (0,5.5) .. controls (2,5) and (3,4.5) .. (4,4);
\draw[thick] (4,4) -- (10,4);
\draw[dotted] (4,0) node[below]{\footnotesize{$\xi_\alpha^1$}}-- (4,4);
\draw[thick] (-10,4) -- (-3,4);
\draw[dotted] (-3,0) node[below]{\footnotesize{$\xi_\alpha^2$}}-- (-3,4);
\draw[thick] (-3,4) .. controls (-2,3) and (-1,2.5) .. (0,2) node[midway,below]{\footnotesize{$\phi_{\alpha,0}$}};
\draw (-0.2,2.3)  node[right]{\footnotesize{$\alpha/2$}};
\draw[thick] (0,2) .. controls (1,1.5) and (3,0.5) .. (7,0);
\draw[thick] (7,0) node[below]{\footnotesize{$\xi_0$}}-- (10,0);

\end{tikzpicture}
\end{center}
\caption{\label{f:2phis}{The profiles $\phi_{1,\alpha}$ and $\phi_{\alpha,0}$.}}
\end{figure}
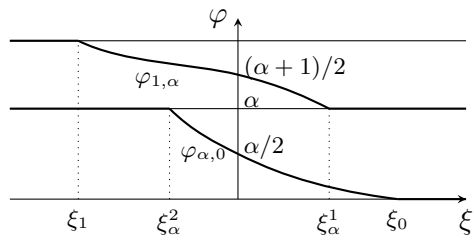

Since the right-hand side of \eqref{e:cauchyEQ} is locally of class $C^1$, then $\phi_{1,\alpha}$ and $\phi_{\alpha,0}$ exist and are unique. Let $\phi_{1,\alpha}$ and $\phi_{\alpha,0}$ be maximally defined in $(\xi_1, \xi_\alpha^1) \subset \R$ and  $(\xi_\alpha^2,\xi_0) \subset \R$, respectively, with
\[
-\infty \le \xi_1 < 0 < \xi_\alpha^1 \le \infty, \quad
-\infty \le \xi_\alpha^2 < 0 < \xi_0 \le \infty,
\]
and satisfying
\[ \lim_{\xi \to \xi_1^+}\phi_{1,\alpha}(\xi)=1, \ \lim_{\xi \to \{\xi_\alpha^1\}^-}\phi_{1,\alpha}(\xi)=\alpha,\]
\[ \lim_{\xi \to \{\xi_\alpha^2\}^+}\phi_{\alpha, 0}(\xi)=\alpha, \ \lim_{\xi \to \xi_0^-}\phi_{\alpha, 0}(\xi)=0.\]
Moreover, from their definitions, we have
\[ \phi_{1,\alpha}\rq{}(\xi)<0 \ \mbox{ if } \ \phi(\xi) \in (\alpha,1) \quad \mbox{ and } \quad  \phi_{\alpha,0}\rq{}(\xi)<0 \ \mbox{ if } \ \phi(\xi) \in (0,\alpha).\]
Since $\phi_{1,\alpha}$ satisfies \eqref{e:ODE} and $z(1)=0$, if $\xi_1 \in \R$ then it can be constantly extended up to $(-\infty, \xi_\alpha^1)$ by setting $\phi_{1,\alpha}(\xi)=1$, for $\xi \le \xi_1$. Analogously, if $\xi_0\in\R$ then $\phi_{\alpha,0}$ can be extended to $(\xi_\alpha^2,\infty)$, by setting $\phi_{\alpha,0}(\xi)=0$, $\xi \ge \xi_0$ (see Definition \ref{d:tws}). In order to glue together $\phi_{1,\alpha}$ and $\phi_{\alpha,0}$ (after space shifts if needed), we need to prove
\begin{equation}
\label{claim}
\xi_\alpha^1 \in\R \ \mbox{ and } \ \xi_\alpha^2 \in \R,
\end{equation}
Consider $\phi_{\alpha,0}$. The strict monotonicity of $\phi_{\alpha,0}$ leads to define a function $\xi:(0,\alpha) \to (\xi_\alpha^2,\xi_0)$ such that $\xi\left(\phi_{\alpha,0}(\xi)\right)=\xi$, for every $\xi \in (\xi_\alpha^2,\xi_0)$, by means of which we have
\begin{equation}
\label{e:xialpha finite}
\xi_\alpha^2=\int_{\frac{\alpha}{2}}^{\alpha}\dot{\xi}(\phi)\,d\phi=\int_{\frac{\alpha}{2}}^{\alpha}\frac{D(\phi)}{z(\phi)}\,d\phi,
\end{equation}
since $\phi_{\alpha,0}(0)=\alpha/2$. Note, $D/z$ is well-defined in $[\frac{\alpha}{2},\alpha)$ and $\xi_\alpha^2 <0$. Hence, to prove that \eqref{e:xialpha finite} is finite we need to study the behavior of $D/z$  in a left-neighborhood of  $\alpha$.
Analogously, by starting from $\phi_{1,\alpha}$ in place of $\phi_{\alpha,0}$, we have that
\begin{equation}
\label{e:xialpha1 finite}
\xi_\alpha^1=-\int_{\alpha}^{(\alpha+1)/2}\frac{D(\phi)}{z(\phi)}\,d\phi.
\end{equation}
Hence, $\xi_\alpha^1 >0$ is finite if and only if it is so the integral of $D/z$ in a right-neighborhood of $\alpha$. Observe that
\[
\lim_{\xi \to \left\{\xi_\alpha^1\right\}^-}D\left(\phi_{1,\alpha}(\xi)\right)\phi_{1,\alpha}\rq{}(\xi)=0 = \lim_{\xi \to \left\{ \xi_\alpha^2\right\}^+} D\left(\phi_{\alpha,0}(\xi)\right)\phi_{\alpha,0}\rq{}(\xi),
\]
because $z$ satisfies $\eqref{e:problem0}_4$. Pasting together $\phi_{1,\alpha}$ and $\phi_{\alpha,0}$ (after shifts if needed) produces a function $\phi$ satisfying \eqref{e:ODE} (in the sense of Definition \ref{d:tws}) in $\R$, associated to the wave speed $c$, and \eqref{e:infty}. Therefore, if we prove that \eqref{claim} holds under $(a.i)$ or $(a.ii)$, then the existence of a profile $\phi$ follows as well.
\smallskip

({\em Proof of \eqref{claim} under $(a.i)$}) We treat separately the cases $\alpha>\gamma$ and $\alpha<\gamma$.

Assume first $\alpha>\gamma$. We note that $D>0$ in $(\gamma,\alpha)$, $D(\alpha)=0$ and $g>0$ in $(\gamma,\alpha]$; moreover, $z$ satisfies \eqref{e:problem0101} in $(\gamma,\alpha)$ and $z(\alpha)=0$. From \cite[formula (9.2)]{BCM1} we deduce that $D(\phi)/z(\phi)$ has finite limit as $\phi \to \alpha^-$. This in turn implies (from \eqref{e:xialpha finite}) that $\xi_\alpha^2 >-\infty$. Indeed, we note that \cite[formula (9.2)]{BCM1} was proved for solutions $z$ vanishing at both the extrema and hence it was assumed $c$ greater than a proper threshold. Nonetheless, the same proof provides \cite[formula (9.2)]{BCM1} in case $z$ vanishing only at the right extremum, and any $c$.  The same argument as above, applied to $\tilde z$ in $(0,\bar\alpha)$, implies $\xi_\alpha^1 \in \R$. In fact, we have $\tilde z(\bar\alpha)=0$ and $\tilde z$ satisfies \eqref{e:problem0101} in $(0,\bar\alpha)$, with $q=\tilde D \bar g$, $\bar h$ and $c$; more importantly, $\tilde D$ and $\bar g$ in $(0,\bar\alpha]$ satisfy the same sign assumptions as $D$ and $g$ in $(0,\alpha]$.

Suppose $\alpha<\gamma$. Note that $\bar D>0$ in $(\bar\alpha,1)$, $\bar D (\bar\alpha)=0$, $\tilde g>0$ in $[\bar\alpha,1)$ and $\bar z$ satisfies \eqref{e:problem0101} with $q=\bar D \tilde g$ and $\tilde h$, $C=-c$. As a consequence, the main assumptions of \cite{CM-DPDE} are satisfied with $(\bar\alpha,1)$ here replacing $(0,\bar\rho)$ there. Then, an application of \cite[proof of Parts $(b)$, $(c)$ and $(d)$ in Theorem 2.5]{CM-DPDE} implies that $\bar D/\bar z$ must be finite limit approaching $\bar \alpha^+$. This in turn implies that  $\xi_\alpha^2 >-\infty$. Similarly, since  $-z$ satisfies \eqref{e:problem0101} in $(\alpha, \gamma)$ with $q=Dg$, $-h$ and $C=-c$, we apply the same arguments as above, based on \cite[proof of Parts $(b)$, $(c)$ and $(d)$ in Theorem 2.5]{CM-DPDE}, to deduce that $\xi_\alpha^1<\infty$.

\smallskip

({\em Proof of \eqref{claim} under $(a.ii)$}) Suppose now $\alpha=\gamma$. Observe, the function $w=\bar z$ in $(\bar\alpha,1)$ satisfies \eqref{e:problem0101} with $q=\bar D \tilde g$, $\tilde h$ and $C=-c$. Moreover, $\bar D >0$ in $(\bar\alpha,1)$, $\bar D(\bar\alpha)=0$, $\tilde g >0$ in $(\bar\alpha,1)$ and $\tilde g(\bar \alpha)=0$. From \cite[Proposition 8.2]{BCM1} with $c^*=-c_{1.1}^*$, we deduce then that if $c=c_{1.1}^*$ (and hence $C=-c_{1.1}^*$) then $\dot w(\bar\alpha^+)<0$. This at once implies that $\dot z(\alpha^-)>0$ and hence that $\xi_\alpha^2 \in \R$. Otherwise, $C> -c_{1.1}^*$ and  $\dot w(\bar\alpha^+)=0$, again from \cite[Proposition 8.2]{BCM1}, because $\dot q(\alpha)=D(\alpha)g(\alpha)=0$. In this case, we argue as follows. Observe that, from \eqref{e:c*below} and \eqref{e:soglie1}, we have $c_{1.1}^*\le h(\alpha)$. Hence, $c=-C<h(\alpha)$. Define
\[ \zeta(\psi):= -k \bar D(\psi) \left(\psi-\bar\alpha\right)^\tau,  \ \psi \in (\bar\alpha,1),\]
where $\tau \in (0,1)$ is the same as in \eqref{integrability g} and $k>0$ is to be determined. We have
\[ \dot \zeta (\psi)=-k \dot{\bar D}(\psi)\left(\psi-\bar\alpha\right)^\tau -\tau k \frac{\bar D(\psi)}{\psi-\bar\alpha} \left(\psi-\bar\alpha\right)^\tau \to 0 \ \mbox{ as } \ \psi \to \bar\alpha^+,\]
while
\begin{equation}
\label{omega1}
\tilde h(\psi)-C-\frac{\bar D(\psi)\tilde g(\psi)}{\omega(\psi)} = \tilde h(\psi)+c +\frac{\tilde g(\psi)}{k\left(\psi-\bar\alpha\right)^\tau} \ge  \tilde h(\psi) +c +\frac{L}{k} \to -h(\alpha)+c+\frac{L}{k},
\end{equation}
in virtue of \eqref{integrability g}. If $k>0$ is chosen (small enough) such that $c-h(\alpha)+L/k >0$, then $\zeta$ is a strict lower-solution of $\eqref{e:problem0101}_1$ in a right-neighborhood of $\bar\alpha$.  On the other hand, since $\dot w(\bar\alpha)=0$ then there exists a sequence $\theta_n \in (\bar\alpha,1)$ such that $\theta_n \to \bar\alpha$ and $\dot w(\theta_n)\to 0$, as $n \to \infty$. This implies that
\begin{equation}
\label{omega2}
\frac{-\zeta(\theta_n)}{-w(\theta_n)} =\frac{k\left(\theta_n-\bar\alpha\right)^\tau}{\tilde g(\theta_n)}\frac{\bar D(\theta_n)\tilde g(\theta_n)}{-w(\theta_n)}\le \frac{k}{L}\frac{\bar D(\theta_n)\tilde g(\theta_n)}{-w(\theta_n)} \to  \frac{k}{L}\left(h(\alpha)-c\right)<1 \ \mbox{ as } \ n\to \infty,
\end{equation}
since $w$ satisfies $\dot w =\tilde h +c -\bar D \tilde g/w$. Then  $-\zeta(\theta_n) < -w(\theta_n)$ for each $n$. Thus, from comparison-type results (see \cite[Lemma 3.2 $(2.b.i)$]{BCM1}), we conclude that $\zeta>w$ in a right-neighborhood of $\bar\alpha$, which means that
\[
z(\phi) < -kD(\phi) \left(\alpha-\phi\right)^\tau \ \mbox{ for $\phi$ in a left-neighborhood of $\alpha$.}
\]
This and \eqref{e:xialpha finite} imply that $\xi_\alpha^2$ is finite.

We now prove that $\xi_\alpha^1$  is finite. We first observe that the function $\psi(\xi)$ defined as $\psi(\xi):=1-\phi_{1,\alpha}(-\xi)$, for $\xi \in (-\xi_\alpha^1,\infty)$, must satisfy
\begin{equation}
\label{e:ODE2}
\left(\tilde D(\psi)\psi\rq{}\right)\rq{} +\left(c-\bar h(\psi)\right)\psi\rq{}+\bar g(\psi)=0,
\end{equation}
 $\psi \in (0,\bar \alpha)$, $\psi(-\xi_\alpha^1)=\bar\alpha$ and $\psi(\infty)=0$. Hence, $\psi$ is a wave profile, connecting $\bar \alpha$ to $0$, with wave speed $c$ associated to \eqref{e:E} with $\tilde D$, $\bar g$ and $\bar h$. Moreover, in virtue of \eqref{integrability g}, a direct check permits to apply \cite[Theorem 2.5 (ii)]{CdRM} and to infer that $\psi$ must be non-strictly monotone, in the sense that it reaches $\bar \alpha$ in a finite value. This clearly means also that $\phi_{1,\alpha}$ reaches $\alpha$ at a finite value. From the very definition of $\xi_\alpha^1$ this implies that $\xi_\alpha^1 \in \R$.


\smallskip
It remains to prove the conclusion of Proposition \ref{prop:suff} under $(b)$. Assume then that  $\alpha<\gamma$ and  $z$ satisfies \eqref{e:problem0prime} with $z(\gamma)=0$. Since, in this case, for $\phi \in (\alpha,1)$ the right-hand side of \eqref{e:cauchyEQ} is necessarily a $C^1$-function only on $(\alpha,\gamma)\cup(\gamma,1)$, we get (by solving \eqref{e:cauchyEQ} with suitable initial conditions) two functions $\phi_{1,\gamma}$ and $\phi_{\gamma,\alpha}$ solving \eqref{e:ODE} with $\phi_{1,\gamma} \in (\gamma,1)$ and $\phi_{\gamma, \alpha}\in (\alpha,\gamma)$. Hence, what we need now is only to prove that $\phi_{1,\gamma}$ and $\phi_{\gamma,\alpha}$ reach $\gamma$ at finite values. To this end, we observe that we can apply a result from \cite{CM-DPDE}. Indeed, it is sufficient to observe that $\phi_{\gamma, \alpha}$ is a wave profile (with speed $c$), connecting $\gamma$ to $\alpha$ for \eqref{e:E} if and only if it is a wave profile (with speed $-c$) for \eqref{e:E} with $-D$, $-g$ and $-h$. Since $-D>0$ in $(\alpha,\gamma)$, $-D(\gamma)>0$, $-g>0$ in $[\alpha,\gamma)$, $-g(\gamma)=0$  and \eqref{integrability g} holds, then in the interval $[\alpha, \gamma]$ we are under the assumptions of \cite[Theorem 2.9 $(ii)$]{CM-DPDE}. As a consequence, $\phi_{\gamma, \alpha}$ must reach $\gamma$ at a finite value. The same argument can be followed to deduce that also $\phi_{1,\gamma}$ must reach $\gamma$ at a finite value, by starting from $\psi(\xi):=1-\phi_{1,\gamma}(-\xi)$, which is a wave profile, connecting $\bar\gamma$ to $0$, of \eqref{e:E} with $\tilde{D}$ and $\bar g$. Since, as before, $\xi_{\alpha}^2$ is finite, pasting together $\phi_{1,\gamma}$, $\phi_{\gamma,\alpha}$ and $\phi_{\alpha,0}$ we conclude the proof.
\end{proof}

\begin{remark}
\label{rem:interval}
{
\rm
Consider the case $\alpha=\gamma$. Proposition \ref{prop:suff} shows that every $c$, for which \eqref{e:problem0} has a solution, is admissible, under assumption \eqref{integrability g}. Without requiring \eqref{integrability g}, we now prove that if $c_1<c_2$ are admissible speeds, then every $c \in (c_1,c_2)$ does.

First, we observe that necessarily $[c_1,c_2]\subseteq [c_{3.2}^*, c_{1.1}^*]$, from Propositions \ref{prop:nec} and \ref{prop:2}. Then, fix $c\in [c_1,c_2]$ and let $z_c$ be the corresponding solution of \eqref{e:problem0}. By using properly Lemma \ref{lem:0101} $(b)$, we have $\bar{z}_{c_2}\le \bar{z}_c \le \bar{z}_{c_1}$ in $(\bar\alpha,1)$ and $\tilde{z}_{c_1} \le \tilde{z}_c \le \tilde{z}_{c_2}$ in $(0,\bar\alpha)$. Hence,
\begin{equation}
\label{zchain}
z_{c_2} \le z_c \le z_{c_1} \ \mbox{ in } \ [0,1].
\end{equation}
By following the proof of Proposition \ref{prop:nec}, we define the functions $\phi_{\alpha,0}$ and $\phi_{1,\alpha}$, from \eqref{e:cauchyEQ}-\eqref{e:cauchyCI}. We can glue $\phi_{\alpha,0}$ and $\phi_{1,\alpha}$ to obtain a profile $\phi$ if and only if \eqref{claim} holds true. However, since $c_1$ and $c_2$ are admissible, then $z_{c_1}$ and $z_{c_2}$ make \eqref{e:xialpha finite}, \eqref{e:xialpha1 finite} finite; by \eqref{zchain}, a direct argument involving also \eqref{e:xialpha finite}, \eqref{e:xialpha1 finite} implies that also $z_c$ makes \eqref{claim} true.
}
\end{remark}

\begin{proof}[Proof of Theorem \ref{th:multiplicity}]
We prove $(a)$. Since $\mathcal J$ is not empty then there exists (at least) a wave profile $\phi$. Corollary \ref{cor:equivalence} $(a)$ implies that this fact is equivalent to the existence of a solution of \eqref{e:problem0}. Proposition \ref{prop:1} informs us that \eqref{e:problem0} is solvable for (at most) a unique $c$ and that the corresponding solution $z$ is unique. Moreover, since $z(\gamma)<0$ and $D(\gamma)>0$, from Lemma \ref{lem:def} and formulas as \eqref{e:zeta1} we have that $\phi\rq{}(\xi)<0$ for every $\xi$ such that $\phi(\xi)=\gamma$. This implies that stretchings at level $\gamma$ as in \eqref{stretch} are not allowed.

We prove $(b)$. By arguing as in case $(a)$, but now from Propositions \ref{prop:nec} $(i)$ and \ref{prop:2}, we obtain that (a non-empty) $\mathcal J$ must be bounded; Remark \ref{rem:interval} implies that $\mathcal{J}$ is an interval. Fix $c \in \mathcal{J}$. By Lemma \ref{lem:def} and $D(\alpha=\gamma)=0$, profiles as \eqref{stretch} are solutions to \eqref{e:ODE}, for every finite $\delta_1,\delta_2$.  Thus, for each $c\in \mathcal J$, the profile must be unique at least up to shifts and stretchings at level $\gamma$; we now prove that no other loss of uniqueness occurs. Suppose by contradiction that, for a given $c$, there exist profiles $\phi_1$ and $\phi_2$ such that $\phi_1 \neq \phi_2$ in $(-\infty, \xi^*)$, $\phi_1(\xi^*)=\phi_2(\xi^*)=\gamma$ and $\phi_1, \phi_2 \in (\gamma,1]$ in $(-\infty,\xi^*)$, for some $\xi^*\in\R$. The one-to-one correspondence between profiles connecting $1$ to $\gamma$ and solutions of \eqref{e:problem0} in $(\gamma,1)$ is standard because $D,g\neq 0$ in $(\gamma,1)$. Since, for a fixed $c$, \eqref{e:problem0} in $(\gamma,1)$ has at most a unique solution $z$ (essentially from Lemma \ref{lem:0101} $(a)$), then $\phi_1=\phi_2$ in $(-\infty, \xi^*)$, a contradiction.  The closedness of $\mathcal J$ under \eqref{integrability g} follows by Corollary \ref{cor:equivalence} $(b)$ and Proposition \ref{prop:2}, directly.

We prove $(c)$. The first part follows analogously to cases $(a)$ and $(b)$, by applying Proposition \ref{prop:nec} $(ii)$ and Remark \ref{rem:problem0prime}. Indeed, if $\phi$ is a profile then $z$ given in the proof of Proposition \ref{prop:nec} (when $\alpha<\gamma$) satisfies \eqref{e:problem0prime} with $z(\gamma)\ge 0$. Remark
\ref{rem:problem0prime} implies that $\mathcal J\subseteq [c_{3,2}^*, \min\{c_{1,1}^*, c_{3,2}^*\}]$, hence it is bounded. By Remark \ref{rem:(c)boundedinterval} and Proposition \ref{prop:3} we get that $\mathcal J \supset (c_{3,2}^*, \min\{c_{1,1}^*, c_{3,2}^*\})$, so it is an interval.  However, we cannot conclude that $\mathcal J$ is also closed because the extremal cases $c=c_{3.1}^*$ and $c=c_{3.2}^*$  are not comprehensively covered, by Proposition \ref{prop:3}. Again by Proposition \ref{prop:3}, for every  $c$ satisfying \eqref{e:range c}, we have the existence of a family of solutions of \eqref{e:problem0} in one-to-one correspondence with $\mathcal Z_c$ in \eqref{e:zc2}. Then, by Proposition \ref{prop:suff} $(a)$, to each of such solutions corresponds a profile (unique up to shifts).  Hence, the existence of the
 family $\{\phi_{\lambda}\}_\lambda$, for  $\lambda\in[\lambda_c,0)$ is given, starting from $z_s$ of \eqref{e:zc2} (with $\lambda=s/D(\gamma)<0$) and $\lambda_c=\beta(c)/D(\gamma)$, where $\beta(c)$ is defined in Lemma \ref{lem:c6final}. The latter part of item ({\em c}) deals with the profiles $\phi_0$. These profiles (up to shifts and stretchings at level $\gamma$), by Corollary \ref{cor:equivalence} $(c)$, must be in one-to-one correspondence with solutions $z$ of \eqref{e:problem0prime} with $z(\gamma)=0$. Hence, from Remark \ref{rem:problem0prime}, we deduce that $\mathcal J=[c_{3,2}^*, \min\{c_{1,1}^*, c_{3,2}^*\}]$.

 The last part, regarding $\phi\rq{}$, is deduced by Proposition \ref{prop:nec}.
\end{proof}

\begin{remark}
\label{nec bounds}
{
\rm
 In the case $\alpha\le \gamma$, a direct consequence of Propositions \ref{prop:nec}, \ref{prop:2} and Remark \ref{rem:problem0prime} is that, if profiles exist, then $c_{3.2}^* \le c \le c_{1.1}^*$. Due to \eqref{e:soglie1} and by applying \eqref{e:c*below} to $c_{1.1}^*$ and $c_{3.2}^*$, then it follows $h(1) \le c \le h(\alpha)$, necessarily.
}
\end{remark}

\begin{proof}[Proof of Theorem \ref{thm:f convex}]
First, we claim that, if $f$ is convex and there exists a wave profile, then $\alpha>\gamma$. We argue by contradiction assuming both that a profile exists and $\alpha\le \gamma$. As observed in Remark \ref{nec bounds}, we have $ c\ge h(1)$. Moreover, to $c_{1.1}^*$ we apply \eqref{e:c*below2}, because $f$ is convex. Hence, we must have $c_{1.1}^*< h(\alpha)$, which in turn implies $h(\alpha)> h(1)$. This contradicts the convexity assumption on $f$. Hence, we have showed the claim.

Assume then $\alpha>\gamma$. From Proposition \ref{prop:nec} $(i)$ there exists a solution of \eqref{e:problem0} and from Proposition \ref{prop:1} $c$ must equal $c_1^*$. Multiplying by $z(\phi)$ and integrating $\eqref{e:problem0}_1$ in $(\alpha,1)$ gives
\[ 0=\int_{\alpha}^{1} (h(\sigma)-c_1^*)z(\sigma)\,d\sigma -\int_{\alpha}^1 D(\sigma)g(\sigma)\,d\sigma,\]
that is
\begin{equation}\label{integral eq} c_1^* \int_{\alpha}^{1}z(\sigma)\,d\sigma=\int_{\alpha}^1 h(\sigma) z(\sigma)\,d\sigma - \int_{\alpha}^1 D(\sigma)g(\sigma)\,d\sigma.\end{equation}
This implies, because of $Dg <0$ and $z>0$ in $(\alpha,1)$,
\[c_1^* > \frac{\int_{\alpha}^{1}h(\sigma)z(\sigma)\,d\sigma}{\int_{\alpha}^1 z(\sigma)\,d\sigma} \ge \min_{[\alpha,1]}h.\]
Since $f$ is convex then
\begin{equation}
\label{e:simple estimate}c_1^* > h(\alpha).
\end{equation}
We prove now \eqref{e:fconv1}. By arguing similarly but in $(0,\alpha)$ we obtain that $c_1^*<h(\alpha)$ if $\int_{0}^{\alpha}Dg <0$ and $c_1^* \le h(\alpha)$ if $\int_{0}^{\alpha}Dg =0$. Thus, we must have \eqref{e:fconv1}.

Now, we prove that \eqref{fconvex suff cond} implies the existence of solutions. First, note that \eqref{fconvex suff cond} implies necessarily $\alpha > \gamma$. By applying Proposition \ref{prop:suff} $(a.i)$ and Proposition \ref{prop:1}, solutions exist if both $c_{1.1}^* < c_{1.2}^*$ and  $c_1^* \ge c_{3.2}^*$. We prove first the former one. Since $f$ is convex, \eqref{e:c*below2} applied to $c_{1.1}^*$ and $c_{1.2}^*$ defined in \eqref{e:soglie1} implies that
\[ c_{1.1}^* <h(\gamma) < c_{1.2}^*.\]
It remains to prove that $c_1^* \ge c_{3.2}^*$. By applying \eqref{e:c*above} to $c_{3.2}^*$ in \eqref{e:soglie1} we have $c_{3.2}^* \le \bar{c}:=h(1) + 2\sqrt{\sup_{[\alpha,1)}\delta(Dg,1)}$. Hence, in order to conclude it suffices to show that $c_1^* \ge \bar c$.

Let $z_1$ be the solution of $\eqref{e:problem2}_1$ corresponding to $c=\bar c$. Note, $z_1$ is well-defined since $\bar c >h(1)\ge h(\gamma) > c_{1.1}^*$ (see Lemma \ref{lem:problem1}). Since $z_1(\gamma)<0$, $z_1$ can be maximally extended up to some $\phi_0 \in (\gamma, \alpha]$ as a unique solution of the differential equation in $\eqref{e:problem2}_1$. Still call $z_1$ this extension. Observe that such a $z_1$ turns out to be continuous up to $[0,\phi_0]$ (see \cite[Lemma 3.1]{BCM1}). It is plain to verify that $\phi_0 < \alpha$ if and only if $c_1^* > \bar c$ (see Figure \ref{f:fz1z2}). We then prove that if \eqref{fconvex suff cond} holds then $\phi_0<\alpha$. To this end, since $z_1(\phi_0)\in (-\infty,0]$ we observe that
\begin{equation}\label{1} 0\le z_1(\phi_0)^2 - z_1(0)^2 =2\int_{0}^{\phi_0}\left(h(\rho)-\sigma\right) z_1(\rho)\,d\rho -2\int_0^{\phi_0} q(\rho)\,d\rho.\end{equation}
Define
\[ \Delta:=\left(\bar c-h(0)\right)\sqrt{\alpha}+\sqrt{\left(\bar c-h(0)\right)^2\alpha +2M}=\Sigma\sqrt{\alpha}+\sqrt{\Sigma^2\alpha +2M}.\]
Since $f$ is convex, by arguing similarly to the proof of \cite[Theorem 3.2]{Maini-Malaguti-Marcelli-Matucci2007}, then $\phi \mapsto -\Delta \phi^{1/2}$ turns out to be a strict lower-solution of $\eqref{e:problem0101}_1$ in $(0,\alpha)$,  with $c=\bar c$. Hence, $z_1(\phi)>-\Delta \phi^{1/2}$ in $(0,\phi_0)$. This, the convexity of $f$, and the very definition of $\bar c$ imply that
\[\int_{0}^{\phi_0}\left(h(\rho)-\bar c\right)z_1(\rho)\,d\rho < \left(\bar c-h(0)\right)\Delta\int_{0}^{\phi_0}\sqrt{\rho}\,d\rho<\Sigma \Delta \int_{0}^{\alpha}\sqrt{\rho}\,d\rho.
\]
From \eqref{1} then we obtain
\[0 < \frac{4}{3}\Sigma\Delta \alpha^{3/2} -2\int_{0}^{\phi_0} q(\rho)\,d\rho.\]
Since $q>0$ in $(\gamma, \alpha)$, if \eqref{fconvex suff cond} holds then $\phi_0 < \alpha$; hence, $c_1^* >\bar c$. This concludes the proof.
\end{proof}

\begin{proof}[Proof of Theorem \ref{th:sufficient conditions}]
We first observe that, since $f$ is assumed to be (strictly) concave then
\[
\inf_{[0,\gamma)}\delta(f, \gamma)=h(\gamma)=\dot f (\gamma), \ \mbox{ and } \ \sup_{[\alpha,1)}\delta(f,1)=\frac{f(1)-f(\alpha)}{1-\alpha}.
\]

Assume $\alpha>\gamma$. In virtue of \eqref{e:soglie1}, from \eqref{e:c*above} applied to $c_{1.1}^*$ and $c_{3.2}^*$ we have
\[ c_{1.1}^* \ge \inf_{[0,\gamma)}\delta(f,\gamma) - 2 \sqrt{\sup_{[0,\gamma)}\delta(Dg, \gamma)} \ge \sup_{[\alpha,1)}\delta(f,1) +2\sqrt{\sup_{[\alpha,1)}\delta(Dg,1)} \ge c_{3.2}^*.\]
where the second inequality is given by \eqref{sufficient condition1}. Moreover, since $\dot g(\gamma)>0$ then $h(\gamma)-2\sqrt{D(\gamma)\dot g(\gamma)}<h(\gamma)+2\sqrt{D(\gamma)\dot g(\gamma)}$ and hence, in virtue of \eqref{e:soglie1}, from \eqref{e:c*below} applied to $c_{1.1}^*$ and $c_{1.2}^*$ we have $c_{1.1}^* < c_{1.2}^*$. Hence, from Proposition \ref{prop:1} we infer that \eqref{e:problem0} is solvable for some $c$. Proposition \ref{prop:suff}  $(a.i)$ then implies that, associated to that $c$, there exists a profile $\phi$.

Assume $\alpha=\gamma$. As in the case $\alpha>\gamma$, but now we apply \eqref{e:c*above2} to $c_{1.1}^*$ and deduce $c_{1.1}^* > c_{3.2}^*$. Indeed, estimate \eqref{e:c*above2} informs us that
\[
c_{1.1}^* \ge  \inf_{[0,\gamma)}\delta(f,\gamma) - 2 \sqrt{\sup_{\phi \in [0,\gamma)} \frac1{\gamma-\phi}\int_{\phi}^{\gamma}\frac{D(\sigma)g(\sigma)}{\sigma-\gamma}\,d\sigma} > \inf_{[0,\gamma)}\delta(f,\gamma) - 2 \sqrt{\sup_{[0,\gamma)}\delta(Dg, \gamma)},
\]
because $\delta(Dg,\gamma)(\sigma)=D(\sigma)g(\sigma)/(\sigma-\gamma)$ in $[0,\gamma)$ is a positive non-constant function which vanishes at $\gamma$ and hence its supremum is strictly larger than that of its mean value function. Proposition \ref{prop:2} implies that \eqref{e:problem0} is solvable (for a non-empty interval of values for $c$ which must satisfy $c\le h(\alpha)$, from \eqref{e:c*below} applied to $c_{1.1}^*$). Since there exists at least a value of such $c$ satisfying $c<h(\alpha)$ (indeed there exists an interval) then, under \eqref{integrability g}, Proposition \ref{prop:suff} $(a.ii)$ gives the existence of profiles.

Assume $\alpha<\gamma$. First, we observe that since $f$ is strictly concave then
\[
 \dot f(\gamma)<\frac{f(\gamma)-f(\alpha)}{\gamma-\alpha}= \inf_{[0,\gamma]\setminus\{\alpha\}}\delta(f,\alpha).
\]
Condition \eqref{sufficient condition2} then implies
\[ \inf_{[0,\gamma]\setminus\{\alpha\}}\delta(f,\alpha)-2 \sqrt{\sup_{[0,\gamma]\setminus\{\alpha\}}\delta\left(Dg, \alpha\right)}>  \sup_{[\gamma,1)}\delta(f,1) + 2 \sqrt{\sup_{[\gamma,1)}\delta\left(Dg,1\right)}.\]
From \eqref{e:soglie1} and \eqref{e:c*above} for $c_{3.2}^*$ we deduce $c_{3.2}^* \le  \sup_{[\gamma,1)}\delta(f,1) + 2 \sqrt{\sup_{[\gamma,1)}\delta\left(Dg,1\right)}$ while  for \eqref{e:c*above} applied to $c_{3.1}^*$ and $c_{1.1}^*$ we obtain
\begin{multline*}
\inf_{[0,\gamma]\setminus\{\alpha\}}\delta(f,\alpha)-2 \sqrt{\sup_{[0,\gamma]\setminus\{\alpha\}}\delta\left(Dg, \alpha\right)} \le \\ \min\{ \inf_{[0,\alpha)}\delta(f,\alpha) -2 \sqrt{\sup_{[0,\alpha)}\delta(Dg,\alpha)},  \inf_{(\alpha, \gamma]}\delta(f,\alpha) -2 \sqrt{\sup_{(\alpha,\gamma])}\delta(Dg,\alpha)}\} \le \min\{c_{1.1}^*, c_{3.1}^*\}.
\end{multline*}
Hence, we deduce that $c_{3.2}^*< \min\{c_{1.1}^*, c_{3.1}^*\}$. Proposition \ref{prop:3} informs us that \eqref{e:problem0} is solvable. Proposition \ref{prop:suff} $(a.i)$ implies in turn that there exist profiles $\phi$.
\end{proof}

\begin{remark}
{
\rm
In the case $\alpha=\gamma$, we proved in Theorem \ref{th:sufficient conditions} that $c_{3.2}^*<c_{1.1}^*$. Hence, in virtue of Corollary \ref{cor:equivalence} and Proposition \ref{prop:2}, the profiles, whose existence is assured by Theorem \ref{th:sufficient conditions}, correspond to a whole interval of speeds, which does not collapse to a single value.

In the case $\alpha<\gamma$, Theorem \ref{th:sufficient conditions} shows {\em en passant} the existence of profiles $\phi_\lambda$ (under the notation of Theorem \ref{th:multiplicity}) for $\lambda \in [\lambda(c), 0)$, with $c$ in the whole interval given by \eqref{e:range c} (which in particular excludes the extremal case $c_{3.2}^*$). For these speeds, a consequence of Proposition \ref{prop:suff} $(b)$ is that, in order to admit also the profile $\phi_0$, we would need to add the assumption \eqref{integrability g}. Moreover, in such a case, also the speed $c=c_{3.2}^*$ is admissible for $\phi_0$ (see Remark \ref{rem:problem0prime}).
}
\end{remark}

\section{Examples}
\label{s:examples}
\setcounter{equation}{0}
In this section we provide some examples concerning specific situations which occur in the above results.

\begin{example}[$\mathcal J$ degenerating to a single point]
\label{ex:Jeqzero}
{
\rm
We show that, when $\alpha\le \gamma$, the interval $\mathcal J$ of Theorem \ref{th:multiplicity} may reduce to a single value.

Assume first $\alpha=\gamma=\frac1{2}$. Let $D$ and $g$ be defined by
\[
D(\phi):= \left( \frac1{2}-\phi\right)\phi \ \mbox{ and } \ g(\phi):=-\phi^2\left(\frac1{2}-\phi\right)^2 \ \mbox{ for } \ \phi \in \left[0,\frac1{2}\right]
\]
and by symmetry as $D(\phi)=-D(\phi-1/2)$ and $g(\phi)=-g(\phi-1/2)$, if $\phi \in (1/2,1]$.
Such functions $D$ and $g$ satisfy (D) and (g) with $\alpha=\gamma=1/2$ and are such that $q=Dg$ satisfies
\[
q(\phi)=q\left(\phi-\frac1{2}\right) \ \mbox{ for } \ \phi \in \left[\frac1{2},1\right].
\]
Also, for some $\tau, \sigma \in (0,1)$, define, for $\phi \in [0,1/2]$,
\[
\ h(\phi):=\phi^\tau\left(\frac1{2}-\phi\right)^\sigma\left\{(1+\sigma)\phi+\phi^{2-2\tau}\left(\frac1{2}-\phi\right)^{2-2\sigma}-(1+\tau)\left(\frac1{2}-\phi\right) \right\},
\]
and such that $h(\phi)=-h(\phi-1/2)$ for each $\phi \in (1/2,1]$. Such a function $h$ is continuous since $h({1/2}^-)=h({1/2}^+)=0$.

In this case, from the symmetry properties of $q$ and $h$ and \eqref{e:soglie1}, it is direct to verify that
\[ c_{3.2}^* = -c_{1.1}^*.\]
We have that the function $z$ defined by
\[
z(\phi):=\begin{cases} -\phi^{1+\tau}\left(\frac1{2}-\phi\right)^{1+\sigma}, \ &\phi \in \left[0,\frac1{2}\right],\\[6pt] \left(1-\phi\right)^{1+\sigma}\left(\phi-\frac1{2}\right)^{1+\tau}, \ &\phi \in \left(\frac1{2},1\right],\end{cases}
\]
satisfies \eqref{e:problem0} with  $c=0$. Notice that $z(\phi) = -z(\phi-1/2)$ if $\phi\in[0,1/2]$. Hence, from Proposition \ref{prop:2}, $c_{3.2}^* \le 0$. Moreover, since from \eqref{e:c*below} we have $c_{3.2}^* \ge h(1)=0$ then necessarily $c_{3.2}^*=0$. This gives $c_{3.2}^*=c_{1.1}^*=0$. From Proposition \ref{prop:nec} $(i)$ we have then
\[ \mbox{either } \mathcal J=\varnothing \ \mbox{ or } \mathcal J=\{0\}.\]
Observe, since \eqref{integrability g} does not hold, in this example, we cannot directly apply Proposition \ref{prop:suff} $(a.ii)$, but looking for profiles associated to the speed $c=0$ we proceed as follows. We need to solve (separately) the two Cauchy problems
\begin{equation*}
\begin{array}{lll}
\phi\rq{}&=\ds\frac{z(\phi)}{D(\phi)}=-\phi^\tau\left(\frac1{2}-\phi\right)^\sigma, &\phi(0)=\phi_0,
\\
\psi\rq{}&=\ds\frac{z(\psi)}{D(\psi)}= -{(1-\psi)^{\sigma}}\left(\psi-\frac1{2}\right)^\tau, &\psi(0)=\psi_0,
\end{array}
\end{equation*}
for arbitrary $0 < \phi_0 < 1/2$ and $1/2< \psi_0<1$. Focus on the problem for $\phi$. From
 \begin{equation}\label{e:cauchyintegral1}\int_{\phi_0}^{\phi}\frac{ds}{-s^\tau\left(\frac1{2}-s\right)^\sigma}=\xi,\end{equation}
we get a profile which connects $1/2$ to $0$ such that $\phi$ reaches $1/2$ at a finite $\xi_*$ given by:
 \[ \xi_*= \int_{\phi_0}^{1/2}\frac{ds}{-s^\tau\left(\frac1{2}-s\right)^\sigma} \in \R.\]
 Analogously, from
 \begin{equation}\label{e:cauchyintegral2}
 \int_{\psi_0}^{\psi}\frac{-ds}{(1-s)^{\sigma}\left(s-\frac1{2}\right)^\tau}=\xi,
 \end{equation}
 we obtain a profile $\psi$ associated to $c=0$ which connects $1$ to $1/2$ and such that it reaches $1/2$ at a finite value. Gluing together $\phi$ and $\psi$ gives a desired wavefront and hence
 \[ \mathcal J=\{0\}.\]
Note, the profile is unique up to space shifts and stretchings at level $1/2$.

\smallskip
Assume now $\alpha=1/3<\gamma=2/3$. Define
\[
D(\phi)= \begin{cases} \left(\phi-\frac1{3}\right)^4, \ &\phi \in \left[0,\frac1{3}\right]\!,\\[4pt]-\left(\phi-\frac1{3}\right)^4, \ &\phi \in \left(\frac1{3},\frac2{3}\right]\!,\\[4pt] \frac4{27}\left(\frac2{3}-\phi\right)-\frac1{81},  &\phi \in \left(\frac2{3},1\right]\!,\end{cases}
\
g(\phi)= \begin{cases}- \beta\phi^{1-\tau}, \ &\phi \in \left[0,\frac1{3}\right]\!,\\[5pt] -\beta\left(\frac2{3}-\phi\right)^{1-\tau}, \ &\phi \in \left(\frac1{3},\frac2{3}\right]\!,\\[5pt] \beta\left(\phi-\frac2{3}\right)^{1-\tau}\frac{\left(\phi-1\right)^4}{\frac4{27}\left(\phi-\frac2{3}\right)+\frac1{81}},  &\phi \in \left(\frac2{3},1\right]\!,\end{cases}\]
 for some $\tau \in (0,1)$ and $\beta=1-\tau/2 \in (0,1)$. Such functions $D$ and $g$ satisfy (D) with $\alpha=1/3$ and (g) with $\gamma=2/3$. Note, $g$ satisfies also \eqref{integrability g}. Despite the lack of symmetry in the definition of $D$ and $g$, the function $q=Dg$ satisfies
\[q(\phi)=-q\left(\frac2{3}-\phi\right) \ \mbox{ for } \ \phi \in \left[\frac1{3},\frac2{3}\right] \ \mbox{ and } \ q(\phi)=q\left(\phi-\frac2{3}\right) \ \mbox{ for } \  \phi \in \left[\frac2{3},1\right].\]
Also, define $h(\phi)=2\phi^\beta \left(1/3-\phi\right)$, for $\phi \in [0,1/3]$. In the other two sub-intervals, we define $h$ by symmetry by $h(\phi)=-h(2/3-\phi)$, for $\phi \in (1/3,2/3]$ and $h(\phi)=-h(\phi-2/3)$, for $\phi \in (2/3,1]$. This makes $h$ continuous. Hence, from the symmetry properties of $q$ and $h$ and from \eqref{e:soglie1}, it holds
\begin{equation}\label{e:exampleee} c_{1.1}^*=c_{3.1}^*=-c_{3.2}^*.\end{equation}
Define the function $z(\phi)$ as
\begin{equation}
\label{e:zz}z(\phi)=\begin{cases}-\phi^\beta \left(\frac1{3}-\phi\right)^2, \ &\phi \in \left[0,\frac1{3}\right],\\[8pt]
\left(\frac2{3}-\phi\right)^\beta\left(\phi-\frac1{3}\right)^2, \ &\phi \in \left(\frac1{3}, \frac{2}{3}\right],\\[8pt]
\left(\phi-\frac{2}{3}\right)^\beta\left(1-\phi\right)^2, \ &\phi \in \left(\frac2{3},1\right],
\end{cases}
\end{equation}
so that $z(\phi)=-z(2/3-\phi)$, for $\phi \in (1/3,2/3]$, and $z(\phi)=-z(\phi-2/3)$, for $\phi \in (2/3,1]$. Direct computations show that $z$ is a solution of \eqref{e:problem0prime}, with $z(\gamma)=0$ and with $c=0$. We highlight that $z$ is of class $C^1$ neither in a neighborhood of $2/3=\gamma$ nor in a right neighborhood of $0$.  From Remark \ref{rem:problem0prime} and \eqref{e:exampleee} we deduce $c_{3.2}^*\le 0 \le -c_{3.2}^*$. From \eqref{e:c*below} applied to $c_{3.2}^*$ we know that $c_{3.2}^* \ge h(1)=0$. We hence obtain
\[ c_{1.1}^*=c_{3.1}^*=c_{3.2}^*=0.\]
By direct computation as in \eqref{e:cauchyintegral1}-\eqref{e:cauchyintegral2}, or by Proposition \ref{prop:suff} $(b)$, we conclude that there exists a profile $\phi$ ($\phi_0$ in the notation of Theorem \ref{th:multiplicity}), which is unique up to shifts and stretchings at level $1/3$.
}
\end{example}

\begin{example}[Solutions with plateaus when $\alpha<\gamma$]
\label{other}
{
\rm
In the case $\alpha<\gamma$, we show that equation \eqref{e:E} can admit wavefronts, connecting $1$ to $0$, such that $\phi\rq{}(\xi_\gamma)=0$ for every $\xi_\gamma$ (possibly not unique) with $\phi(\xi_\gamma)=\gamma$. Hence, $\phi=\phi_0$ under the notation of Theorem \ref{th:multiplicity} $(c)$. This can occur only if \eqref{g sublinear gamma} does not hold, that is when either $D^+g(\gamma)=\infty$ or $D^-g(\gamma)=\infty$ ; under \eqref{g sublinear gamma} we have $\phi\rq{}(\xi_\gamma)<0$. We refer to \cite[Example 1]{MMM2004} for a similar example in the case $D=1$.
\par
For simplicity, we set $\alpha=0$. Define $D$ and $g$ in $[0,1]$,
\[ D(\phi):=-\phi^2 \ \mbox{ and } \ g(\phi):=\frac{\sigma+1}{2}\left(\phi-\gamma\right)\left|\phi-\gamma\right|^{\sigma-1}\left(1-\phi\right)^2,\]
for some $\sigma \in (0,1)$ and $\gamma\in (0,1)$. Notice that $g$ is not Lipschitz-continuous at point $\gamma$. We want to define a strict (at $0$) semi-wavefront from $1$ to $0$.  We have $D<0$ in $(0,1]$ and $D(0)=0$ while $g<0$ in $[0,\gamma)$, $g>0$ in $(\gamma,1)$ and $g(1)=0$. Direct computations show that the function $z(\phi):=\phi\left|\phi-\gamma\right|^\frac{\sigma+1}{2}(1-\phi)$ satisfies, for $\phi \in (0,\gamma)\cup(\gamma,1)$,
\begin{equation}
\label{c0}
\dot{z}(\phi)=h(\phi)-\frac{D(\phi)g(\phi)}{z(\phi)},
\end{equation}
 with $h(\phi):=\left(1-2\phi\right)\left|\phi-\gamma\right|^\frac{\sigma+1}{2}$. We have $z>0$ in $(0,\gamma)\cup(\gamma,1)$ and $z(0)=z(\gamma)=z(1)=0$. Moreover, $\dot{z}(\gamma^-)=-\infty$ and $\dot{z}(\gamma^+)=\infty$. Also, we have
\[ \int^{\gamma} \frac{D(s)}{z(s)}\,ds>-\infty \ \mbox{ and } \ \int_{\gamma} \frac{D(s)}{z(s)}>-\infty.\]
\par
Let $\phi_{\gamma,0}$ be the solution of
\begin{equation}
\label{cauchy1111}
\phi\rq{}=\frac{z(\phi)}{D(\phi)}=\frac{\left(\phi-1\right)\left|\gamma-\phi\right|^\frac{\sigma+1}{2}}{\phi}, \ \phi(0)=\phi_0\in (0,\gamma).
\end{equation}
Such a solution exists and is unique since $G\rq{}<0$ in $(0,\gamma)$, where
\[
G(t):=
\int_{\phi_0}^{t} \frac{s}{(s-1)|\gamma-s|^\frac{\sigma+1}{2}}\,ds,
\]
so that $\phi_{\gamma,0}(\xi)=G^{-1}(\xi)$ and $\phi_{\gamma,0} \in (0,\gamma)$. Analogously, because $G\rq{}<0$ also in $(\gamma,1)$, it results (uniquely) defined $\phi_{1,\gamma}$ as the solution of \eqref{cauchy1111} with $\phi_0 \in (\gamma,1)$. The function $\phi_{1,\gamma}$ is such that  $\phi_{1,\gamma} \in (\gamma,1)$ and $\phi_{1,\gamma}(\xi)=G^{-1}(\xi)$. Define
\[
\xi_1:=G(1^-), \
 \xi_\gamma^1:=G(\gamma^+), \
\xi_\gamma^2:=G(\gamma^-) \ \mbox{ and } \ \xi_0=G(0^+),
\]
such that $\phi_{1,\gamma}$ is maximally defined in $(\xi_1,\xi_\gamma^1)$ and $\phi_{\gamma,0}$ is maximally defined in $(\xi_\gamma^2, \xi_0)$. It is easy to check that
\[
\xi_1=-\infty, \ \xi_\gamma^1, \xi_\gamma^2, \xi_0\in\R,
\]
and that
\[
\phi_{1,\gamma}\rq{}({\xi_\gamma^1}^-)=\phi_{\gamma,0}\rq{}({\xi_\gamma^2}^+)=0, \ \phi_{\gamma,0}\rq{}({\xi_0}^-)=-\infty.
\]
Thus, it is possible to glue together $\phi_{1,\gamma}$ and $\phi_{\gamma,0}$ to obtain $\phi$, the profile of a strict (at $0$) semi-wavefront which connects $1$ to $0$.
\par
For simplicity, we set $\alpha=1$. We consider then the functions
\[ D(\phi):=\left(\phi-1\right)^2 \ \mbox{ and } \ g(\phi):=-\tau\phi,\]
for some $\tau>0$. We want now to define a profile of a strict (at $1$) wavefront from $1$ to $0$. We have that the function $z(\phi):=\lambda\phi\left(\phi-1\right)$,
where $\lambda>0$, satisfies \eqref{c0} in $(0,1)$, with $h(\phi):=\lambda\left(2\phi-1\right)-\frac{\tau(\phi-1)}{\lambda}$, such that $h(1)=\lambda$. In this case, by solving
\[ \phi\rq{}=\frac{z(\phi)}{D(\phi)}=\frac{\lambda \phi}{\phi-1}, \ \phi(0)=\frac1{2},\]
we obtain that
\[
\frac1{\lambda}\int_{\frac1{2}}^{\phi} \frac{s-1}{s}\,ds=\xi.
\]
This gives a function $\psi$, defined in its maximal-existence interval $(\xi_1, \xi_0)$, which turns out to be the profile of a strict (at $1$) semi-wavefront, which connects $1$ to $0$. Also, $\psi$ is such that
 $\psi\rq{}(\xi_1)=-\infty$. Note, $\xi_1 \in \R$ because $\xi_1=\frac1{\lambda}\int_{1/2}^{1}\frac{s-1}{s}ds \in \R$.
\par
Finally, pasting together manipulations of $\phi$ and $\psi$ we have provided the desired example. Hence, plateaus are possible.
}
\end{example}

\begin{example}[Non-regular $g$ and no plateaus]
\label{other2}
{
\rm
Again for $\alpha<\gamma$, we continue the discussion of Example \ref{other} by showing that $D^+ g(\gamma)=\infty$ (or $D^-g(\gamma)=\infty$) does not necessarily imply $\phi\rq{}(\xi_\gamma)=0$, as Example \ref{other} could suggest. More precisely, under the above condition, we exhibit a profile connecting $1$ to $0$ such that  $\phi\rq{}(\xi_\gamma)<0$ in the unique $\xi_\gamma$ such that $\phi(\xi_\gamma)=\gamma$. Hence, in this case, plateaus are not possible. Differently from the previous example, a slightly different choice of $g$ and $D$ (and $h$) yields a solution $z$ that does not vanish at $\gamma$.

 We set $\alpha=0$. We want to define a strict (at $0$) semi-wavefront from $1$ to $0$. Define $D$ and $g$ in $[0,1]$ by
\begin{equation}
\label{D-g}
 D(\phi)=-\phi \ \mbox{ and } \ g(\phi):=\left|\phi-\gamma\right|^{\sigma-1}\left(\phi-\gamma\right)\left(1-\phi\right),
 \end{equation}
where $\sigma \in (0,1)$.  We have $D<0$ in $(0,1]$ and $D(0)=0$, while $g<0$ in $[0,\gamma)$, $g>0$ in $(\gamma,1)$ and $g(1)=0$. Then the function $z(\phi):=\phi\left(1-\phi\right)$
satisfies \eqref{c0} with $h(\phi):=1-2\phi-\left|\phi-\gamma\right|^{\sigma-1}\left(\phi-\gamma\right)$. Clearly, $z>0$ in $(0,1)$ and $z(0)=z(1)=0$. The corresponding profile such that $\phi(0)=1/2$ is given explicitly by
\[
\phi(\xi)=1-\frac{e^\xi}{2}, \ \xi \in \left(-\infty, \log(2)\right).
\]
By arguing similarly to the second part of Example \ref{other}, we conclude the example.
}
\end{example}


\appendix
\section{Appendix}
\label{s:appendix}
\setcounter{equation}{0}

 In this appendix we compute the derivative of a profile $\phi$ at points where it reaches $\alpha$. For $\alpha\neq \gamma$, we denote $\xi_\alpha \in \R$ the unique value such that $\phi(\xi_\alpha)=\alpha$, while for $\alpha=\gamma$ we define $\xi_\alpha^1:=\sup\{\xi: \phi(\xi)>\alpha\}\le\xi_\alpha^2:=\inf\{\xi: \phi(\xi)<\alpha\}$.
For $c\in \R$, we define $s_\pm(c)$ by
\[
2 s_\pm(c):={h(\alpha)\pm c-\sqrt{\left(h(\alpha)-c\right)^2-4 \dot{D}(\alpha)g(\alpha)}}.
\]
Recall that $\dot{D}(\alpha)\le0$. We show below that $\phi$ has negative (right or left) derivative at $\xi_\alpha$, which can be infinite only if $\dot{D}(\alpha)=0$. To stress the latter behavior, for $Q\in\R$ we denote
\[
[Q]_{-\infty}:= \left\{
\begin{array}{ll}
Q & \hbox{ if }\dot{D}(\alpha)<0,
\\
-\infty & \hbox{ if }\dot{D}(\alpha)=0.
\end{array}
\right.
\]
\begin{proposition}
\label{prop:reg alpha2}
Let $\phi$ be a profile satisfying \eqref{e:infty} with speed $c$.
We have:
\begin{enumerate}[(a)]
\item if $\alpha>\gamma$ then
\begin{equation}
\label{e:reg alpha 1}
\phi\rq{}({\xi_\alpha}^\pm)=
\left\{
\begin{array}{ll}
\frac{g(\alpha)}{s_-(c)} \ &\mbox{ if } \ \dot{D}(\alpha)<0 \ \mbox{ or } \ c>h(\alpha),
\\
\ds-\infty \ &\mbox{ if } \ \dot{D}(\alpha)=0 \ \mbox{ and } \ c\le h(\alpha);
\end{array}
\right.
\end{equation}

\item if $\alpha=\gamma$ then
\begin{enumerate}[(i)]
\item if $c<c_{1.1}^*$ then $\phi\rq{}({\xi_\alpha^2}^+)= 0$;

\item if $c=c_{1.1}^*<h(\alpha)$, then
$  \phi\rq{}({\xi_\alpha^2}^+)=
\left[\frac{h(\alpha)-c_{1.1}^*}{{\dot{D}(\alpha)}}\right]_{-\infty}$;

\item if $c=h(\alpha)$ and $\dot{D}(\alpha)<0$ then $\phi\rq{}({\xi_\alpha^1}^-)=0$;

\item if $c<h(\alpha)$, then
  $\phi\rq{}({\xi_\alpha^1}^-)=
\left[\frac{h(\alpha)-c}{{\dot{D}(\alpha)}}\right]_{-\infty}$;
\end{enumerate}

\item
if $\alpha <\gamma$ then
\begin{enumerate}[(i)]
\item if $c< \min\{c_{1.1}^*, c_{3.1}^*\}$, then $\phi\rq{}({\xi_\alpha}^\pm)= \frac{ g(\alpha)}{s_+(c)}$;

\item if $c=c_{1.1}^*=c_{3.1}^*$, then
$\phi\rq{}({\xi_\alpha}^\pm)=
\left[\frac{g(\alpha)}{s_-(c)}\right]_{-\infty};$

\item if $c=c_{1.1}^*<c_{3.1}^*$, then
$\phi\rq{}({\xi_\alpha}^+) = \frac{g(\alpha)}{s_+(c)}>
\phi\rq{}({\xi_\alpha}^-)=
\left[\frac{g(\alpha)}{s_-(c)}\right]_{-\infty}$;

\item if $c=c_{3.1}^*<c_{1.1}^*$, then
$\phi\rq{}({\xi_\alpha}^-)= \frac{g(\alpha)}{s_+(c)} >
\phi\rq{}({\xi_\alpha}^+)=
\left[\frac{g(\alpha)}{s_-(c)}\right]_{-\infty}$.
\end{enumerate}

\end{enumerate}
\end{proposition}

\begin{proof}
We prove $(a)$. In the interval $(\gamma , 1)$, $D$ changes sign, from positive to negative, and $g>0$. Then, we obtain \eqref{e:reg alpha 1} from \cite[formula (2.9)]{BCM2}.

We prove $(b)$. Let $\phi_{\alpha,0}$ be the restriction of $\phi$ to $(\xi_\alpha, \infty)$, for $\xi_\alpha:=\xi_\alpha^2$; hence, $\phi_{\alpha,0}\in(0,\alpha)$ and $D>0$, $g<0$ there. The function $\psi_{1,\bar\alpha}(\xi):=1-\phi_{\alpha,0}(-\xi)$, for $\xi \in (-\infty, -\xi_\alpha)$, satisfies
\begin{equation}
\label{e:eq psi}
\left(\bar D(\psi)\psi\rq{}\right)\rq{}+\left(-c-\tilde h(\psi)\right)\psi\rq{}+\tilde g(\psi)=0,
\end{equation}
with $\bar D>0$, $\tilde g>0$ in $(\bar\alpha, 1)$. Since $\tilde g(\bar\alpha)=\bar D(\bar\alpha)=0$, by \cite[Theorem 2.3]{BCM1} we deduce the behavior of $\psi_{1,\bar\alpha}$ near $-\xi_\alpha$, where $\psi(-\xi_\alpha^-)=\bar\alpha$; the threshold $c^*$ appearing there is replaced here by $-c_{1.1}^*$, since $-c_{1.1}^*=c^*(\tilde q; \tilde h; \bar\alpha,1)$. Hence, \cite[Theorem 2.3 (i)]{BCM1} implies that $\psi$ is classical at $\bar\alpha$ if $-c>-c_{1.1}^*$, i.e., $\psi\rq{}(-\xi_\alpha^-)=0$;
this proves $(b. i)$, since $\phi\rq{}(\xi)=\psi\rq{}(-\xi)$. Similarly, from \cite[Theorem 2.3 (ii)]{BCM1}, we have that if $-c=-c_{1.1}^*>\tilde h\left(\bar\alpha\right)$ then
\[ \psi\rq{}(-\xi_\alpha^-)=
\begin{cases}
\frac{\tilde h(\bar\alpha) +c_{1.1}^*}{\dot{\bar D}(\bar\alpha)} \ &\mbox{ if } \ \dot{\bar D}(\bar\alpha)>0,\\
-\infty \ &\mbox{ if } \ \dot{\bar D}(\bar\alpha)=0,
\end{cases}
\]
which gives $(b. ii)$. Consider the restriction $\phi_{1,\alpha}$ of $\phi$ to $(-\infty, \xi_\alpha^1)$.
Then $\phi_{1,\alpha}\in(\alpha,1)$, where $D<0$ and $g>0$.
Thus, the function $\psi_{\bar\alpha,0}(\xi):=1-\phi_{1,\alpha}(-\xi)$, for $\xi \in (-\xi_\alpha^1, \infty)$, is valued in $(0,\bar \alpha)$ and satisfies \eqref{e:ODE2}, with $\tilde D >0$, $\bar g>0$, in $(0,\bar \alpha)$ and $\tilde D(\bar \alpha)=\bar g(\bar \alpha)=0$.
We apply \cite[Theorem 2.3]{CdRM} to $\psi_{\bar\alpha,0}$; then $\psi\rq{}(\xi)\to 0$ as $\xi \to {-\xi_\alpha^1}^+$ if either $c>\bar h(\bar\alpha)$ or $c=\bar h (\bar\alpha)$ and $\dot{\tilde{D}}(\bar\alpha)<0$. This gives $(b. iii)$. Assume $c<\bar h(\bar\alpha)$ and consider $w(\psi):=\tilde D(\psi)\psi\rq{}$, for $\psi \in (0,\bar\alpha)$. Then $w<0$ in $(0,\bar\alpha)$, $w(\bar\alpha)=w(0)=0$ and $\dot w = \bar h -c - [\tilde D \bar g]/w$ in $(0,\bar\alpha)$, with $\dot {\{ \tilde D \bar g \}}\ \ (\alpha)=0$. Hence, \cite[formula  (7.2)]{BCM1} implies $\dot w(\bar \alpha)=\bar h(\bar \alpha)-c$, since $c<\bar h(\bar\alpha)$ by assumption. This gives
\[ \psi\rq{} ({-\xi_\alpha^1}^+)=
\begin{cases}
\frac{\bar h(\bar\alpha)-c}{\dot{\tilde D}(\bar\alpha)} \ &\mbox{ if } \ \dot{\tilde D}(\bar\alpha)<0,\\
-\infty \ &\mbox{ if } \ \dot{\tilde D}(\bar\alpha)=0,
\end{cases}
\]
from which $(b. iv)$ follows.
\par
We prove $(c)$. Let $\xi_\gamma$ be such that $\phi(\xi_\gamma)=\gamma$. Define $\phi_{\gamma,0}(\xi):=\phi(\xi)$ for $\xi \in (\xi_\gamma, \infty)$ and
$\psi_{1,\bar\gamma}(\xi) :=1-\phi_{\gamma,0}(-\xi)$ for $\xi \in (-\infty,-\xi_\gamma)$.
Part $(c)$ follows from \cite[(2.16)]{BCM2}, where $c_{n,r}^*$, $c_{p,l}^*$ there are replaced by (respectively) $-c_{3.1}^*$, $-c_{1.1}^*$. Indeed, $\psi_{1,\bar\gamma}$ satisfies \eqref{e:eq psi}, $\bar D$ satisfies \cite[${\rm (D_{np})}$]{BCM2} in $[\bar \gamma,1]$ for $\beta=\bar\alpha$, $\tilde g$ satisfies \cite[$(g)$]{BCM2} in $[\bar\gamma,1]$  and $\tilde h$ replaces $h$.
\end{proof}

As an example of degenerate situation, we now give conditions such that \eqref{e:ODE} has {\em only} profiles with infinity slope at $\alpha$, hence sharp at $\alpha$. We stress the role played by the change of convexity of $f$.

\begin{corollary}
\label{cor:slope alpha}
Assume $\alpha>\gamma$. Let $f$ be convex in $(\gamma,\alpha)$ and strictly concave both in $(0,\gamma)$ and $(\alpha,1)$. Moreover, assume $\dot{D}(\alpha)=0$, \eqref{sufficient condition1} and
\begin{equation}\label{e:cor alpha}
\dot f(\alpha) -\frac{f(\alpha)-f(\gamma)}{\alpha-\gamma}  \ge 2\sqrt{\sup_{(\gamma,\alpha]}\delta\left(Dg, \gamma\right)}.
\end{equation}
Then, the unique (up to shifts) profile $\phi$ of Equation \eqref{e:E} satisfies
$\phi\rq{}({\xi_\alpha}^\pm)=-\infty$.
\end{corollary}

\begin{proof}
First, we show that \eqref{e:E} has solutions, and then that the solution (unique from Theorem \ref{th:multiplicity} $(a)$) corresponds to some $c \le h(\alpha)$. We then conclude by Proposition \ref{prop:reg alpha2} $(a)$.

About the existence, since $f$ is convex in $(\gamma,\alpha)$ then it follows from \eqref{e:soglie1} and \eqref{e:c*below}-\eqref{e:c*below2} that $c_{1.1}^*< h(\gamma) \le c_{1.2}^*$ (as observed in the proof of Theorem \ref{thm:f convex}). So $c_1^* \in (c_{1.1}^*,c_{1.2}^*)$ exists. Since $f$ is strictly concave in $(0,\gamma)$, $(\alpha, 1)$, and \eqref{sufficient condition1} holds, by arguing as in the proof of Theorem \ref{th:sufficient conditions} $(a)$ we deduce $c_{1.1}^* \ge c_{3.2}^*$, which implies $c_1^* >c_{3.2}^*$. By applying Corollary \ref{cor:equivalence} $(a)$ and Proposition \ref{prop:1}, Equation \eqref{e:E} admits a profile associated to $c_1^*$.

To prove $c_1^* \le h(\alpha)$, we notice that \eqref{e:cor alpha} becomes, under the assumptions on $f$,
\[ \sup_{(\gamma,\alpha]}\delta\left(f,\gamma\right)+2\sqrt{\sup_{(\gamma,\alpha]}\delta\left(Dg,\gamma\right)} \le \dot f(\alpha)=h(\alpha). \]
By means of \eqref{e:soglie1} and \eqref{e:c*above}, this implies $c_{1.2}^* \le h(\alpha)$ and hence $c_1^*< h(\alpha)$.
\end{proof}

\begin{remark}
{
\rm
Proposition \ref{prop:reg alpha2} can give {\em a posteriori} informations on the value of thresholds, as the following two explicit examples show.

Consider Example \ref{other2}, where $\alpha<\gamma$; we constructed a profile $\phi$ with $\phi\rq{}({\xi_\alpha}^\pm)=-\infty$ and $c=0$. By Proposition \ref{prop:reg alpha2} $(c)$, this can occur only if $c=c_{1.1}^*=c_{3.1}^*$. Then $c_{1.1}^*=c_{3.1}^*=0$.

The profile $\phi$ in Example \ref{example1} has $c=0$ and satisfies $\phi\rq{}({\xi_\alpha^1}^-)=\phi\rq{}({\xi_\alpha^2}^+)<0$. Moreover, $h(\alpha)>0$. In the case $\alpha=\gamma$ we deduce $c=c_{1.1}^*$ by Proposition \ref{prop:reg alpha2} $(b)$; hence $c_{1.1}^*=0$.
}
\end{remark}

\section*{Acknowledgments}
The authors are members of the {\em Gruppo Nazionale per l'Analisi Matematica, la Probabilit\`{a} e le loro Applicazioni} (GNAMPA) of the {\em Istituto Nazionale di Alta Matematica} (INdAM) and acknowledge financial support from this institution.

{\small
\bibliography{refe_BCM3}
\bibliographystyle{abbrv2}
}

\end{document}